\documentclass[reqno]{amsart}
\usepackage{amssymb,amsmath,hyperref}
\usepackage{amsrefs, float}
\usepackage[foot]{amsaddr}
\usepackage{bbold,stackrel}
 
\newcommand{\Z}{{\textsf{\textup{Z}}}}

\newtheorem{thm}{Theorem}

\newtheorem{cor}[thm]{Corollary}
\newtheorem{defi}[thm]{Definition}
\newtheorem{rem}[thm]{Remark}
\newtheorem{nota}[thm]{Notation}
\newtheorem{exa}[thm]{Example}

\newtheorem{princ}[thm]{Principle}

\newtheorem{ack}[thm]{Acknowledgement}

\newtheorem*{tempo*}{Template}
\newtheorem{theorem}[thm]{Theorem}

\newtheorem{lemma}[thm]{Lemma}

\newtheorem{corollary}[thm]{Corollary}

\newcommand\be{\begin{equation}}
\newcommand\ee{\end{equation}} 

\usepackage{amsmath,amsfonts} 

\usepackage[applemac]{inputenc}

\def\bdefi{\begin{defi}\rm}
\def\edefi{\end{defi}}
\def\bnota{\begin{nota}\rm}
\def\enota{\end{nota}}
\def\FIVE{\Pi_{1}^{1}\text{-\textup{\textsf{CA}}}_{0}}

\def\SIXK{\Pi_{k}^{1}\text{-\textsf{\textup{CA}}}_{0}^{\omega}}

\def\ATR{\textup{\textsf{ATR}}}

\def\ZFC{\textup{\textsf{ZFC}}}
\def\ZF{\textup{\textsf{ZF}}}

\def\NCC{\textup{\textsf{NCC}}}
\def\L{\textsf{\textup{L}}}

 \def\r{\mathbb{r}}

\def\RCA{\textup{\textsf{RCA}}}
\def\({\textup{(}}
\def\){\textup{)}}

\def\c{\textup{\textsf{c}}}

\def\RCAo{\textup{\textsf{RCA}}_{0}^{\omega}}
\def\ACAo{\textup{\textsf{ACA}}_{0}^{\omega}}
\def\WKL{\textup{\textsf{WKL}}}

\def\bye{\end{document}}
\def\N{{\mathbb  N}}
\def\Q{{\mathbb  Q}}
\def\R{{\mathbb  R}}

\def\SS{\textup{\textsf{S}}}

\def\di{\rightarrow}

\def\asa{\leftrightarrow}

\def\ACA{\textup{\textsf{ACA}}}

\def\QFAC{\textup{\textsf{QF-AC}}}
\def\AC{\textup{\textsf{AC}}}

\def\NIN{\textup{\textsf{NIN}}}
\def\NBI{\textup{\textsf{NBI}}}

\def\CH{\textup{\textsf{CH}}}
\def\MCC{\textup{\textsf{MCC}}}
\def\cont{\textup{\textsf{coco}}}
\def\LOC{\textup{\textsf{LOC}}}

\def\URY{\textup{\textsf{URY}}}
\def\TIET{\textup{\textsf{TIE}}}
\def\HBC{\textup{\textsf{HBC}}}

\def\open{\textup{\textsf{open}}}

\def\CBT{\textup{\textsf{CBT}}}

\def\u{\textup{\textsf{u}}}

\def\BOOT{\textup{\textsf{BOOT}}}

\def\INDD{\textup{\textbf{IND}}}
\def\NFP{\textup{\textsf{NFP}}}

\def\HBU{\textup{\textsf{HBU}}}

\def\PIT{\textup{\textsf{PIT}}}

\def\w{\textup{\textsf{w}}}
\def\p{\textup{\textsf{p}}}

\def\RM{\textup{\textsf{rm}}}

\def\DCA{\Delta\textup{\textsf{-CA}}}

\def\eps{\varepsilon}

\def\ECF{\textup{\textsf{ECF}}}

\def\SCF{\textup{\textsf{SFF}}}

\usepackage{graphicx}
\usepackage{tikz}
\usetikzlibrary{matrix, shapes.misc}

\numberwithin{equation}{section}
\numberwithin{thm}{section}

\usepackage{comment,tikz-cd}

\begin{document}
\title[Axiom of Choice in Computability Theory and Reverse Mathematics]{The Axiom of Choice in Computability Theory and Reverse Mathematics \\ {\tiny with a cameo for the Continuum Hypothesis}}
\author{Dag Normann}
\address{Department of Mathematics, The University 
of Oslo, P.O. Box 1053, Blindern N-0316 Oslo, Norway}
\email{dnormann@math.uio.no}
\author{Sam Sanders}
\address{Department of Mathematics, TU Darmstadt, Darmstadt, Germany}
\email{sasander@me.com}
\begin{abstract}
The \emph{Axiom of Choice} ($\AC$ for short) is the most (in)famous axiom of the usual foundations of mathematics, $\ZFC$ set theory.  
The (non-)essential use of $\AC$ in mathematics has been well-studied and thoroughly classified. 
Now, fragments of \emph{countable $\AC$} {not} provable in $\ZF$ have recently been used in Kohlenbach's higher-order \emph{Reverse Mathematics} to obtain equivalences between closely related compactness and local-global principles.  
We continue this study and show that $\NCC$, a weak choice principle provable in $\ZF$ and much weaker systems, suffices for many of these results.  
In light of the intimate connection between Reverse Mathematics and computability theory, we also study \emph{realisers} for $\NCC$, i.e.\ functionals that produce 
the choice functions claimed to exist by the latter from the other data.  
Our \emph{hubris} of undertaking the hitherto underdeveloped study of the computational properties of (choice functions from) $\AC$ leads to interesting results.  
For instance, using Kleene's S1-S9 computation schemes, we show that various \emph{total} realisers for $\NCC$ compute Kleene's $\exists^{3}$, a functional that gives rise to full second-order arithmetic, and vice versa.  By contrast, \emph{partial} realisers for $\NCC$ \emph{should}  be much weaker, but establishing this conjecture remains elusive. 
By way of \emph{catharsis}, we show that the Continuum Hypothesis ($\CH$ for short) is equivalent to the existence of a \emph{countably based} partial realiser for $\NCC$.  
The latter kind of realiser does \emph{not} compute Kleene's $\exists^{3}$ and is therefore strictly weaker than a total one.  
\end{abstract}


\maketitle
\thispagestyle{empty}

\section{Introduction}
Obviousness, much more than beauty, is in the eye of the beholder.  For this reason, lest we be misunderstood, we formulate a blanket caveat: all notions (computation, continuity, function, open set, comprehension, et cetera) used in this paper are to be interpreted via their well-known definitions in higher-order arithmetic listed below, \emph{unless explicitly stated otherwise}.  
\subsection{Short summary}
The usual foundations of mathematics \emph{Zermelo-Fraenkel set theory with the Axiom of Choice} and its acronym $\ZFC$, explicitly reference a single axiom.  
The (in)essential use of the Axiom of Choice ($\AC$ for short) in mathematics, is well-studied and has been classified in detail (\cites{heerlijk, howrude,rudebin, rudebin2}).
In a nutshell, this paper deals with the (in)essential use of $\AC$ in Kohlenbach's higher-order \emph{Reverse Mathematics} (RM for short; see Section \ref{prelim1}), and the 
study of the computational properties of the associated fragments of $\AC$ following Kleene's S1-S9 computation schemes (see Section \ref{HCT}).
Our \emph{hubris} of undertaking the hitherto underdeveloped study of the computational properties of choice functions from $\AC$ leads to \emph{catharsis} in that the latter 
properties turn out to be intimately connected to Cantor's \emph{Continuum Hypothesis}, even in the most basic case.  

\smallskip

In more detail, fragments of \emph{countable} $\AC$ \emph{not provable in} $\ZF$, play a central role in the RM of local-global principles and compactness principles in \cite{dagsamV, dagsamVII}.  
The latter principles are generally believed to be intimately related (see e.g.\ Tao's description in \cite{taokejes}*{p.~168}), but they can have \emph{very} 
different logical and computational properties, especially in the \emph{absence} of countable $\AC$, as shown in \cite{dagsamV, dagsamVII} and discussed in detail below in Section \ref{pisec}.

\smallskip

In this paper, we show that countable $\AC$ can be replaced by the much weaker principle $\NCC$ (see Section~\ref{pisec}) provable in higher-order arithmetic \emph{without choice}, and hence $\ZF$.  
Following the intimate connection between RM and computability theory, we also study the computational properties of $\NCC$.
A central role is played by the distinction between \emph{total} and \emph{partial} realisers of $\NCC$. 
Intuitively, the former are \emph{strong} as they compute Kleene's $\exists^{3}$ (yielding full second-order arithmetic; see Section \ref{HCT}), while there \emph{should} be weak examples of the latter that in particular do not compute $\exists^{3}$.
Establishing the latter fact, we run into the famous \emph{Continuum Hypothesis} ($\CH$ for short).  
We explain the required background from \cites{dagsamV, dagsamVII} in Sections \ref{qstn} and \ref{pisec}, while the latter also sketches our main results. 

\smallskip

Finally, $\ZF$ can prove certain choice principles and we refer to those as \emph{weak} fragments of $\AC$, whereas \emph{strong} fragments are those not provable in $\ZF$.
\subsection{Overview} 
We discuss the starting point of this paper, namely \emph{Reverse Mathematics}, in Section \ref{qstn}, while our main results are summarised in Section \ref{pisec}.
\subsubsection{A question with multiple answers}\label{qstn}
The starting point of our enterprise is the \emph{Main Question} of the \emph{Reverse Mathematics} program (RM hereafter; see Section~\ref{prelim1} for an introduction), which is usually formulated as follows.
\begin{quote}
What are the minimal axioms needed to prove a given theorem of ordinary, i.e.\ non-set theoretic mathematics?  (see \cite{simpson2}*{I.1})
\end{quote}
Implicit in this question is the assumption that one can always find a \emph{unique and unambiguous} set of such minimal axioms.  
As it turns out, there are basic theorems for which this question \textbf{does not} have an unique or unambiguous answer.  
The most basic example is \emph{Pincherle's theorem}, published around 1882 in \cite{tepelpinch}*{p.\ 67} and studied in \cite{dagsamV}.
This \emph{third-order} theorem expresses that a locally bounded function is bounded, say on Cantor space for simplicity. 

\smallskip

As discussed in detail in Section \ref{pisec}, and with definitions in Section~\ref{prelim1}, Pincherle's theorem is equivalent to \emph{weak K\"onig's lemma} from second-order RM, over Kohlenbach's base theory $\RCAo$ \textbf{plus} $\QFAC^{0,1}$; the latter is a strong fragment of countable choice.  This equivalence is expected as compactness principles and local-global principles are intimately related in light of Tao's description in \cite{taokejes}*{p.~168}.   
By contrast, in the \emph{absence} of countable choice, there are two conservative extensions of second-order arithmetic $\Z_{2}$, called $\Z_{2}^{\omega}$ and $\Z_{2}^{\Omega}$, where the former \emph{cannot} prove Pincherle's theorem and the latter \emph{can} (and hence $\ZF$ can too).  We note that $\Z_{2}^{\omega}$ is based on \emph{third-order} functionals $\SS_{k}^{2}$ deciding second-order $\Pi_{k}^{1}$-formulas, while $\Z_{2}^{\Omega}$ is based on Kleene's \emph{fourth-order} quantifier $\exists^{3}$.

\smallskip

Similar results are available for the computability theory (in the sense of Kleene's S1-S9 from \cite{kleene2, longmann}): weak K\"onig's lemma is equivalent to the Heine-Borel theorem for countable covers and the
finite sub-cover claimed to exist by the latter is outright computable in terms of the data.  Despite this equivalence and the similar syntactic form, no type two functional (which includes the aforementioned $\SS_{k}^{2}$) can compute 
the upper bound from Pincherle's theorem in terms of the data.

\smallskip

More results of the above nature can be found in \cite{dagsamV, dagsamVII}, as discussed in Section~\ref{pisec}.
Together, these results show that local-global principles (like Pincherle's theorem) are very similar to compactness (like weak K\"onig's lemma), \emph{yet} can behave very differently, esp.\ in the absence of countable choice. 
In the spirit of RM, it is then a natural question whether countable choice is necessary in this context, or whether a weak(er) choice principle, say provable in $\ZF$, suffices.  
A positive answer is provided in the next section, as well 
as the implications for (higher-order) computability theory.  Indeed, the latter is intimately connected to RM, prompting the study of realisers for the aforementioned weak choice principles.  

\smallskip

Finally, we note that, in the grander scheme of things, there is (was?) a movement to remove countable choice from Bishop's constructive analysis \cite{BISHNOCC, BISHNOCC2, shoeshoe, BRS} and constructive mathematics (\cite{troeleke2}*{\S3.9}).  
While classical, our results do fit with the spirit of this constructive enterprise. 
\subsubsection{The Pincherle phenomenon}\label{pisec}
We formulate the results from \cite{dagsamV, dagsamVII}, the \emph{Pincherle phenomenon} in particular, and sketch our results based on this phenomenon.

\smallskip

First of all, we have shown in \cite{dagsamV} that Pincherle's theorem is \emph{closely related} to (open-cover) compactness, but has \emph{fundamentally different} logical and computational properties.  
Indeed, Pincherle's theorem, called $\PIT_{o}$ in \cite{dagsamV}, satisfies the following properties; definitions can be found in Section \ref{HCT} and \ref{hijo}.    
\begin{itemize}
\item[(I)] The systems $\Z_{2}^{\omega}$ and $\Z_{2}^{\Omega}$ are conservative extensions of $\Z_{2}$ and $\Z_{2}^{\omega}$ cannot prove $\PIT_{o}$ while $\Z_{2}^{\Omega}$ can; $\RCAo+\QFAC^{0,1}$ proves $\WKL\asa \PIT_{o}$.
\item[(II)] Even a weak\footnote{Two kinds of realisers for Pincherle's theorem were introduced in \cite{dagsamV}: a \emph{weak Pincherle realiser} $M_{o}$ takes as input $F^{2}$ that is locally bounded on $2^{\N}$ together with $G^{2}$ such that $G(f)$ is an upper bound for $F$ in $[\overline{f}G(f)]$ for any $f\in 2^{\N}$, and outputs an upper bound $M_{o}(F, G)$ for $F$ on $2^{\N}$. A (normal) \emph{Pincherle realiser} $M_{\u}$ outputs an upper bound $M_{\u}(G)$ without access to $F$.  We discuss these functionals in some detail in Section \ref{labbe}.\label{fookie}} realiser for $\PIT_{o}$ cannot be computed (Kleene S1-S9) in terms of any type two functional, including the comprehension functionals $\SS_{k}^{2}$. 
\end{itemize}
Secondly, we have established similar properties in \cite{dagsamVII} for many basic theorems pertaining to open sets given by (possibly discontinuous) characteristic functions.  
A number of results in \cite{samph} also make use of $\QFAC^{0,1}$ in (what seems like) an essential way.  
For instance, let $\HBC$ be the Heine-Borel theorem for countable covers of closed sets in $[0,1]$ which are complements of the aforementioned kind of open sets.  
Exactly the same properties as in items (I) and (II) hold for $\HBC$, and a large number of similar theorems, by \cite{dagsamVII}*{\S3}.

\smallskip

We shall therefore say that \emph{$\HBC$ exhibits the Pincherle phenomemon}, due to Pincherle's theorem $\PIT_{o}$ being the first theorem identified as exhibiting the behaviour as in (I) and (II), namely in \cite{dagsamV}. In other words, the aim of \cite{dagsamVII} was to establish the abundance of the Pincherele phenomenon in ordinary mathematics, beyond the few examples from \cite{dagsamV}.

\smallskip

Thirdly, since $\ZF$ cannot prove $\QFAC^{0,1}$, it is a natural question, also implied by the Main Question of RM, whether a choice principle weaker than $\QFAC^{0,1}$ also suffices to obtain equivalences like $\HBC\asa \WKL\asa \PIT_{o}$.
In this paper, we show that a number of such results originally proved using $\QFAC^{0,1}$, can be proved using the following weak choice principle.  
\bdefi[$\NCC$]
For $Y^{2}$ and $A(n, m)\equiv (\exists f\in 2^{\N})(Y(f, m, n)=0)$:
\[
(\forall n^{0})(\exists m^{0})A(n,m)\di  (\exists g^{1})(\forall n^{0})A(n,g(n)). 
\]
\edefi
Clearly, this principle is provable in $\ZF$ and even in $\Z_{2}^{\Omega}$, a conservative extension of $\Z_{2}$ introduced in Section \ref{HCT}.
The replacement of $\QFAC^{0,1}$ by $\NCC$ is for the most part non-trivial and introduces a lot more technical detail, as will become clear in Section \ref{FRAC}.
An obvious RM-question is whether one can weaken $\NCC$, e.g.\ by letting $g^{1}$ only provide an \emph{upper bound} for the variable $m^{0}$ in $\NCC$. 
The below proofs do not seem to go through with this modification.  As discussed in Section~\ref{basik}, $\NCC$ is also connected to the \emph{uncountability of $\R$} in interesting ways. 

\smallskip

Finally, since RM and computability theory are generally intimately connected (both in the second- and higher-order case), it is a natural next step to study the computational properties of $\NCC$, 
even though choice functions provided by $\AC$ are often regarded as fundamentally non-constructive. 
We study \emph{realisers} for $\NCC$, which are functionals $\zeta$ that take as input $Y$ and output the choice function $\zeta(Y) =g$ from $\NCC$.
While $\NCC$ is quite weak, the associated realisers turn out to be rather \emph{strong}, in that they compute the aforementioned $\exists^{3}$, a functional that yields full second-order arithmetic (and vice versa).  
We establish the same for \emph{weak} realisers for $\NCC$ that only yield an \emph{upper bound} for the variable $m^{0}$ from $\NCC$.  

\smallskip

Finally, the strength of the aforementioned realisers is due to their \emph{total} nature, and it is therefore natural to study \emph{partial}, i.e.\ not everywhere defined, realisers for $\NCC$.  
In particular, we believe these realisers to be the key to answering the following question raised in \cites{dagsamV}.  
Intuitively speaking, there should be a difference between 
the following two computational problems (A) and (B).
\begin{itemize}
\item[(A)] For any $G:2^{\N}\di \N$, compute a finite sub-cover of $\cup_{f\in 2^{\N}}[\overline{f}G(f)]$, i.e.\ compute $f_{1}, \dots, f_{k}\in 2^{\N}$ such that $\cup_{i\leq k}[\overline{f_{i}}G(f_{i})]$ covers $2^{\N}$.
\item[(B)] For any $G:2^{\N}\di \N$, compute a number $k\in \N$ such that \textbf{there exists} a finite sub-cover $f_{1}, \dots, f_{k}\in 2^{\N}$ of $\cup_{f\in 2^{\N}}[\overline{f}G(f)]$.
\end{itemize}
The problem (A) gives rise to $\Theta$-functionals, introduced in Section \ref{HCT}, while (B) gives rise to realisers for `uniform' Pincherle's theorem, introduced in Section \ref{labbe}.
Note that in item (A), one needs to provide elements in Cantor space (which can code infinitely much information), while item (B) only requires a natural number (which can only code finite information).  
In Section \ref{incel}, we show that partial realisers for $\NCC$ can perform (B); we \emph{conjecture} that they cannot perform (A). 

\smallskip

Since we do not have any idea how to establish the aforementioned conjecture, we shall solve a weaker\footnote{It is shown in \cite{dagsam} that $\exists^{3}$ can perform the computational task (A).} problem, namely finding a partial realiser of $\NCC$ that does not compute Kleene's $\exists^{3}$.   
In this context, the property \emph{countably based}, a kind of higher-order continuity property as in Definition \ref{frag}, is helpful.  Indeed, countably based functionals cannot compute $\exists^{3}$, i.e.\ 
a countably based realiser for $\NCC$ is just what we want.  
Much to our surprise, this kind of construct does exist, but is rather elusive as the following is proved in Theorem \ref{thm.CH}.  
\begin{center}
The Continuum Hypothesis $\CH$ is equivalent to the existence of a \emph{countably based} partial realiser for $\NCC$.  
\end{center}  
This result perhaps constitutes a kind of \emph{catharsis} following the \emph{hubris} of 
studying the computational properties of choice functions from $\AC$.   
Entertaining as this equivalence may be, it would be preferable to have a $\ZFC$-proof of the existence of a partial realiser for $\NCC$ that does not compute $\exists^{3}$.

\section{Preliminaries}\label{prelim}

We introduce \emph{Reverse Mathematics} in Section \ref{prelim1}, as well as its generalisation to \emph{higher-order arithmetic}, and the associated base theory $\RCAo$.  
We introduce some essential axioms in Section~\ref{HCT}.  

\subsection{Reverse Mathematics}\label{prelim1}
Reverse Mathematics is a program in the foundations of mathematics initiated around 1975 by Friedman (\cites{fried,fried2}) and developed extensively by Simpson (\cite{simpson2}).  
The aim of RM is to identify the minimal axioms needed to prove theorems of ordinary, i.e.\ non-set theoretical, mathematics. 

\smallskip

We refer to \cite{stillebron} for a basic introduction to RM and to \cite{simpson2, simpson1} for an overview of RM.  We expect basic familiarity with RM, but do sketch some aspects of Kohlenbach's \emph{higher-order} RM (\cite{kohlenbach2}) essential to this paper, including the base theory $\RCAo$ (Definition \ref{kase}).  
As will become clear, the latter is officially a type theory but can accommodate (enough) set theory. 

\smallskip

First of all, in contrast to `classical' RM based on \emph{second-order arithmetic} $\Z_{2}$, higher-order RM uses $\L_{\omega}$, the richer language of \emph{higher-order arithmetic}.  
Indeed, while the former is restricted to natural numbers and sets of natural numbers, higher-order arithmetic can accommodate sets of sets of natural numbers, sets of sets of sets of natural numbers, et cetera.  
To formalise this idea, we introduce the collection of \emph{all finite types} $\mathbf{T}$, defined by the two clauses:
\begin{center}
(i) $0\in \mathbf{T}$   and   (ii)  If $\sigma, \tau\in \mathbf{T}$ then $( \sigma \di \tau) \in \mathbf{T}$,
\end{center}
where $0$ is the type of natural numbers, and $\sigma\di \tau$ is the type of mappings from objects of type $\sigma$ to objects of type $\tau$.
In this way, $1\equiv 0\di 0$ is the type of functions from numbers to numbers, and  $n+1\equiv n\di 0$.  Viewing sets as given by characteristic functions, we note that $\Z_{2}$ only includes objects of type $0$ and $1$.    

\smallskip

Secondly, the language $\L_{\omega}$ includes variables $x^{\rho}, y^{\rho}, z^{\rho},\dots$ of any finite type $\rho\in \mathbf{T}$.  Types may be omitted when they can be inferred from context.  
The constants of $\L_{\omega}$ include the type $0$ objects $0, 1$ and $ <_{0}, +_{0}, \times_{0},=_{0}$  which are intended to have their usual meaning as operations on $\N$.
Equality at higher types is defined in terms of `$=_{0}$' as follows: for any objects $x^{\tau}, y^{\tau}$, we have
\be\label{aparth}
[x=_{\tau}y] \equiv (\forall z_{1}^{\tau_{1}}\dots z_{k}^{\tau_{k}})[xz_{1}\dots z_{k}=_{0}yz_{1}\dots z_{k}],
\ee
if the type $\tau$ is composed as $\tau\equiv(\tau_{1}\di \dots\di \tau_{k}\di 0)$.  
Furthermore, $\L_{\omega}$ also includes the \emph{recursor constant} $\mathbf{R}_{\sigma}$ for any $\sigma\in \mathbf{T}$, which allows for iteration on type $\sigma$-objects as in the special case \eqref{special}.  Formulas and terms are defined as usual.  
One obtains the sub-language $\L_{n+2}$ by restricting the above type formation rule to produce only type $n+1$ objects (and related types of similar complexity).        
\bdefi\label{kase} 
The base theory $\RCAo$ consists of the following axioms.
\begin{enumerate}
 \renewcommand{\theenumi}{\alph{enumi}}
\item  Basic axioms expressing that $0, 1, <_{0}, +_{0}, \times_{0}$ form an ordered semi-ring with equality $=_{0}$.
\item Basic axioms defining the well-known $\Pi$ and $\Sigma$ combinators (aka $K$ and $S$ in \cite{avi2}), which allow for the definition of \emph{$\lambda$-abstraction}. 
\item The defining axiom of the recursor constant $\mathbf{R}_{0}$: for $m^{0}$ and $f^{1}$: 
\be\label{special}
\mathbf{R}_{0}(f, m, 0):= m \textup{ and } \mathbf{R}_{0}(f, m, n+1):= f(n, \mathbf{R}_{0}(f, m, n)).
\ee
\item The \emph{axiom of extensionality}: for all $\rho, \tau\in \mathbf{T}$, we have:
\be\label{EXT}\tag{$\textsf{\textup{E}}_{\rho, \tau}$}  
(\forall  x^{\rho},y^{\rho}, \varphi^{\rho\di \tau}) \big[x=_{\rho} y \di \varphi(x)=_{\tau}\varphi(y)   \big].
\ee 
\item The induction axiom for quantifier-free\footnote{To be absolutely clear, variables (of any finite type) are allowed in quantifier-free formulas of the language $\L_{\omega}$: only quantifiers are banned.} formulas of $\L_{\omega}$.
\item $\QFAC^{1,0}$: the quantifier-free Axiom of Choice as in Definition \ref{QFAC}.
\end{enumerate}
\edefi
\noindent
We let $\INDD^{\omega}$ be the induction axiom for all formulas in $\L_{\omega}$.
\bdefi\label{QFAC} The axiom $\QFAC$ consists of the following for all $\sigma, \tau \in \textbf{T}$:
\be\tag{$\QFAC^{\sigma,\tau}$}
(\forall x^{\sigma})(\exists y^{\tau})A(x, y)\di (\exists Y^{\sigma\di \tau})(\forall x^{\sigma})A(x, Y(x)),
\ee
for any quantifier-free formula $A$ in the language of $\L_{\omega}$.
\edefi
As discussed in \cite{kohlenbach2}*{\S2}, $\RCAo$ and $\RCA_{0}$ prove the same sentences `up to language' as the latter is set-based and the former function-based.  Recursion as in \eqref{special} is called \emph{primitive recursion}; the class of functionals obtained from $\mathbf{R}_{\rho}$ for all $\rho \in \mathbf{T}$ is called \emph{G\"odel's system $T$} of all (higher-order) primitive recursive functionals.  

\smallskip

We use the usual notations for natural, rational, and real numbers, and the associated functions, as introduced in \cite{kohlenbach2}*{p.\ 288-289}.  
\begin{defi}[Real numbers and related notions in $\RCAo$]\label{keepintireal}\rm~
\begin{enumerate}
 \renewcommand{\theenumi}{\alph{enumi}}
\item Natural numbers correspond to type zero objects, and we use `$n^{0}$' and `$n\in \N$' interchangeably.  Rational numbers are defined as signed quotients of natural numbers, and `$q\in \Q$' and `$<_{\Q}$' have their usual meaning.    
\item Real numbers are coded by fast-converging Cauchy sequences $q_{(\cdot)}:\N\di \Q$, i.e.\  such that $(\forall n^{0}, i^{0})(|q_{n}-q_{n+i}|<_{\Q} \frac{1}{2^{n}})$.  
We use Kohlenbach's `hat function' from \cite{kohlenbach2}*{p.\ 289} to guarantee that every $q^{1}$ defines a real number.  
\item We write `$x\in \R$' to express that $x^{1}:=(q^{1}_{(\cdot)})$ represents a real as in the previous item and write $[x](k):=q_{k}$ for the $k$-th approximation of $x$.    
\item Two reals $x, y$ represented by $q_{(\cdot)}$ and $r_{(\cdot)}$ are \emph{equal}, denoted $x=_{\R}y$, if $(\forall n^{0})(|q_{n}-r_{n}|\leq {2^{-n+1}})$. Inequality `$<_{\R}$' is defined similarly.  
We sometimes omit the subscript `$\R$' if it is clear from context.           
\item Functions $F:\R\di \R$ are represented by $\Phi^{1\di 1}$ mapping equal reals to equal reals, i.e.\ extensionality as in $(\forall x , y\in \R)(x=_{\R}y\di \Phi(x)=_{\R}\Phi(y))$.\label{EXTEN}
\item The relation `$x\leq_{\tau}y$' is defined as in \eqref{aparth} but with `$\leq_{0}$' instead of `$=_{0}$'.  Binary sequences are denoted `$f^{1}, g^{1}\leq_{1}1$', but also `$f,g\in C$' or `$f, g\in 2^{\N}$'.  Elements of Baire space are given by $f^{1}, g^{1}$, but also denoted `$f, g\in \N^{\N}$'.
\item For a binary sequence $f^{1}$, the associated real in $[0,1]$ is $\r(f):=\sum_{n=0}^{\infty}\frac{f(n)}{2^{n+1}}$.\label{detrippe}
\item An object $\textbf{Y}^{0\di \rho}$ is called \emph{a sequence of type $\rho$ objects} and also denoted $\textbf{Y}=(Y_{n})_{n\in \N}$ or $\textbf{Y}=\lambda n. Y_{n}$ where $Y_{n}:=\textbf{Y}(n)$ for all $n^{0}$.
\end{enumerate}
\end{defi}
\noindent
Below, we shall discuss various \emph{different} notions of open set, namely as in Definitions~\ref{openset} and~\ref{opensset}.
Hence, we do not provide a general definition of set here.  
Next, we mention the highly useful $\ECF$-interpretation. 
\begin{rem}[The $\ECF$-interpretation]\label{ECF}\rm
The (rather) technical definition of $\ECF$ may be found in \cite{troelstra1}*{p.\ 138, \S2.6}.
Intuitively, the $\ECF$-interpretation $[A]_{\ECF}$ of a formula $A\in \L_{\omega}$ is just $A$ with all variables 
of type two and higher replaced by type one variables ranging over so-called `associates' or `RM-codes' (see \cite{kohlenbach4}*{\S4}); the latter are (countable) representations of continuous functionals.  
The $\ECF$-interpretation connects $\RCAo$ and $\RCA_{0}$ (see \cite{kohlenbach2}*{Prop.\ 3.1}) in that if $\RCAo$ proves $A$, then $\RCA_{0}$ proves $[A]_{\ECF}$, again `up to language', as $\RCA_{0}$ is 
formulated using sets, and $[A]_{\ECF}$ is formulated using types, i.e.\ using type zero and one objects.  
\end{rem}
In light of the widespread use of codes in RM and the common practise of identifying codes with the objects being coded, it is no exaggeration to refer to $\ECF$ as the \emph{canonical} embedding of higher-order into second-order arithmetic. 
For completeness, we list the following notational convention for finite sequences.  
\begin{nota}[Finite sequences]\label{skim}\rm
We assume a dedicated type for `finite sequences of objects of type $\rho$', namely $\rho^{*}$.  Since the usual coding of pairs of numbers goes through in $\RCAo$, we shall not always distinguish between $0$ and $0^{*}$. 
Similarly, we do not always distinguish between `$s^{\rho}$' and `$\langle s^{\rho}\rangle$', where the former is `the object $s$ of type $\rho$', and the latter is `the sequence of type $\rho^{*}$ with only element $s^{\rho}$'.  The empty sequence for the type $\rho^{*}$ is denoted by `$\langle \rangle_{\rho}$', usually with the typing omitted.  

\smallskip

Furthermore, we denote by `$|s|=n$' the length of the finite sequence $s^{\rho^{*}}=\langle s_{0}^{\rho},s_{1}^{\rho},\dots,s_{n-1}^{\rho}\rangle$, where $|\langle\rangle|=0$, i.e.\ the empty sequence has length zero.
For sequences $s^{\rho^{*}}, t^{\rho^{*}}$, we denote by `$s*t$' the concatenation of $s$ and $t$, i.e.\ $(s*t)(i)=s(i)$ for $i<|s|$ and $(s*t)(j)=t(|s|-j)$ for $|s|\leq j< |s|+|t|$. For a sequence $s^{\rho^{*}}$, we define $\overline{s}N:=\langle s(0), s(1), \dots,  s(N-1)\rangle $ for $N^{0}<|s|$.  
For a sequence $\alpha^{0\di \rho}$, we also write $\overline{\alpha}N=\langle \alpha(0), \alpha(1),\dots, \alpha(N-1)\rangle$ for \emph{any} $N^{0}$.  By way of shorthand, 
$(\forall q^{\rho}\in Q^{\rho^{*}})A(q)$ abbreviates $(\forall i^{0}<|Q|)A(Q(i))$, which is (equivalent to) quantifier-free if $A$ is.   
\end{nota}

\subsection{Higher-order computability theory}\label{HCT}
As noted above, some of our main results are part of computability theory.
Thus, we first make our notion of `computability' precise as follows.  
\begin{itemize}
\item We adopt $\ZFC$, i.e.\ Zermelo-Fraenkel set theory with the Axiom of Choice, as the official metatheory for all results, unless explicitly stated otherwise.
\item We adopt Kleene's notion of \emph{higher-order computation} as given by his nine clauses S1-S9 (see \cite{longmann}*{Ch.\ 5} or \cite{kleene2}) as our official notion of `computable'. 
\end{itemize}
We discuss our choice of framework, and a possible alternative, in Section~\ref{optional}. 

\smallskip

Secondly, similar to \cites{dagsam,dagsamII, dagsamIII, dagsamV, dagsamVI}, one main aim of this paper is the study of functionals of type 3 that are \emph{natural} from the perspective of mathematical practise. 
Our functionals are \emph{genuinely} of type 3 in the sense that they are not computable from any functional of type 2.  The following definition is standard in this context. 
\bdefi\label{frag}\rm
A functional $\Phi^{3}$ is \emph{countably based} if for every $F^{2}$ there is countable $X\subset \N^{\N}$ such that $\Phi(F) = \Phi(G)$ for every $G$ that agrees with $F$ on $X$.
\edefi
Stanley Wainer (unpublished) has defined the countably based functionals of finite type as an analogue of the continuous functionals, while 
John Hartley has investigated the computability theory of this type structure in \cite{hartleycountable}. 

\smallskip

We only use countably based functionals of type at most $3$ in this paper.
Now, if $\Phi^{3}$ is computable in a functional of type 2, then it is countably based, but the converse does not hold. However, Hartley proves in \cite{hartleycountable} that, assuming $\ZFC+\CH$ however, if $\Phi^3$ is not countably based, then there is some $F^{2}$ such that $\exists^3$ (see below) is computable in $\Phi$ and $F$.  In other words, stating the existence of a non-countably based $\Phi$ brings us `close to' $\Z_2^\Omega$ (defined below). 
In the sequel, we shall explicitly point out where we use countably based functionals.  

\smallskip

The importance of Definition \ref{frag} can be understood as follows: to answer whether a given functional $\Phi^{3}$ can compute another functional $\Psi^{3}$, the answer is automatically `no' if $\Phi$ is countably based and $\Psi$ is not. 
A similar `rule-of-thumb' is that if $\Phi$ does not compute $\exists^{2}$ (or a discontinuous functional on $\R$ or $\N^{\N}$; see below), while $\Psi$ does, the answer is similarly `no'.
We have used both rules-of-thumb throughout our project to provide a first `rough' classification of new functionals. 

\smallskip

For the rest of this section, we introduce some existing functionals which will be used below.
In particular, we introduce some functionals which constitute the counterparts of second-order arithmetic $\Z_{2}$, and some of the Big Five systems, in higher-order RM.
We use the formulation from \cite{kohlenbach2, dagsamIII}.  

\smallskip
\noindent
First of all, $\ACA_{0}$ is readily derived from:
\begin{align}\label{mu}\tag{$\mu^{2}$}
(\exists \mu^{2})(\forall f^{1})\big[ (\exists n)(f(n)=0) \di [(f(\mu(f))=0)&\wedge (\forall i<\mu(f))f(i)\ne 0 ]\\
& \wedge [ (\forall n)(f(n)\ne0)\di   \mu(f)=0]    \big], \notag
\end{align}
and $\ACA_{0}^{\omega}\equiv\RCAo+(\mu^{2})$ proves the same sentences as $\ACA_{0}$ by \cite{hunterphd}*{Theorem~2.5}.   The (unique) functional $\mu^{2}$ in $(\mu^{2})$ is also called \emph{Feferman's $\mu$} (\cite{avi2}), 
and is clearly \emph{discontinuous} at $f=_{1}11\dots$; in fact, $(\mu^{2})$ is equivalent to the existence of $F:\R\di\R$ such that $F(x)=1$ if $x>_{\R}0$, and $0$ otherwise (\cite{kohlenbach2}*{\S3}), and to 
\be\label{muk}\tag{$\exists^{2}$}
(\exists \varphi^{2}\leq_{2}1)(\forall f^{1})\big[(\exists n)(f(n)=0) \asa \varphi(f)=0    \big]. 
\ee
\noindent
Secondly, $\FIVE$ is readily derived from the following sentence:
\be\tag{$\SS^{2}$}
(\exists\SS^{2}\leq_{2}1)(\forall f^{1})\big[  (\exists g^{1})(\forall n^{0})(f(\overline{g}n)=0)\asa \SS(f)=0  \big], 
\ee
and $\FIVE^{\omega}\equiv \RCAo+(\SS^{2})$ proves the same $\Pi_{3}^{1}$-sentences as $\FIVE$ by \cite{yamayamaharehare}*{Theorem 2.2}.   The (unique) functional $\SS^{2}$ in $(\SS^{2})$ is also called \emph{the Suslin functional} (\cite{kohlenbach2}).
By definition, the Suslin functional $\SS^{2}$ can decide whether a $\Sigma_{1}^{1}$-formula as in the left-hand side of $(\SS^{2})$ is true or false.   We similarly define the functional $\SS_{k}^{2}$ which decides the truth or falsity of $\Sigma_{k}^{1}$-formulas; we also define 
the system $\SIXK$ as $\RCAo+(\SS_{k}^{2})$, where  $(\SS_{k}^{2})$ expresses that $\SS_{k}^{2}$ exists.  Note that we allow formulas with \emph{function} parameters, but \textbf{not} \emph{functionals} here.
In fact, Gandy's \emph{Superjump} (\cite{supergandy}) constitutes a way of extending $\FIVE^{\omega}$ to parameters of type two.  We identify the functionals $\exists^{2}$ and $\SS_{0}^{2}$ and the systems $\ACAo$ and $\SIXK$ for $k=0$.
We note that the operators $\nu_{n}$ from \cite{boekskeopendoen}*{p.\ 129} are essentially $\SS_{n}^{2}$ strengthened to return a witness (if existant) to the $\Sigma_{k}^{1}$-formula at hand.  

\smallskip

\noindent
Thirdly, full second-order arithmetic $\Z_{2}$ is readily derived from $\cup_{k}\SIXK$, or from:
\be\tag{$\exists^{3}$}
(\exists E^{3}\leq_{3}1)(\forall Y^{2})\big[  (\exists f^{1})(Y(f)=0)\asa E(Y)=0  \big], 
\ee
and we therefore define $\Z_{2}^{\Omega}\equiv \RCAo+(\exists^{3})$ and $\Z_{2}^\omega\equiv \cup_{k}\SIXK$, which are conservative over $\Z_{2}$ by \cite{hunterphd}*{Cor.\ 2.6}. 
Despite this close connection, $\Z_{2}^{\omega}$ and $\Z_{2}^{\Omega}$ can behave quite differently, as discussed in e.g.\ \cite{dagsamIII}*{\S2.2}.   The functional from $(\exists^{3})$ is also called `$\exists^{3}$', and we use the same convention for other functionals.  
Note that $(\exists^{3})\asa [(\exists^{2})+(\kappa_{0}^{3})]$ as shown in \cite{samsplit,dagsam}, where the latter is comprehension on $2^{\N}$: 
\be\tag{$\kappa_{0}^{3}$}
(\exists \kappa_{0}^{3}\leq_{3}1)(\forall Y^{2})\big[\kappa_{0}(Y)=0\asa (\exists f\in C)(Y(f)=0)  \big].
\ee  
Other `splittings' are studied in \cite{samsplit}, including $(\kappa_{0}^{3})$.  

\smallskip

Fourth, the Heine-Borel theorem states the existence of a finite sub-covering for an open covering of certain spaces. 
Now, a functional $\Psi:\R\di \R^{+}$ gives rise to the \emph{canonical covering} $\cup_{x\in I} I_{x}^{\Psi}$ for $I\equiv [0,1]$, where $I_{x}^{\Psi}$ is the open interval $(x-\Psi(x), x+\Psi(x))$.  
Hence, the uncountable covering $\cup_{x\in I} I_{x}^{\Psi}$ has a finite sub-covering by the Heine-Borel theorem; in symbols:
\be\tag{$\HBU$}
(\forall \Psi:\R\di \R^{+})(\exists  y_{1}, \dots, y_{k}\in I)(\forall x\in I)( x\in \cup_{i\leq k}I_{y_{i}}^{\Psi}).
\ee
Note that $\HBU$ is almost verbatim \emph{Cousin's lemma} (\cite{cousin1}*{p.\ 22}), i.e.\ the Heine-Borel theorem restricted to canonical coverings.  
This restriction does not make a big difference, as shown in \cite{sahotop}.
By \cite{dagsamIII, dagsamV}, $\Z_{2}^{\Omega}$ proves $\HBU$ but $\Z_{2}^{\omega}+\QFAC^{0,1}$ cannot, 
and basic properties of the \emph{gauge integral} (\cite{zwette, mullingitover}) are equivalent to $\HBU$.  

\smallskip

Fifth, since Cantor space (denoted $C$ or $2^{\N}$) is homeomorphic to a closed subset of $[0,1]$, the former inherits the same property.  
In particular, for any $G^{2}$, the corresponding `canonical covering' of $2^{\N}$ is $\cup_{f\in 2^{\N}}[\overline{f}G(f)]$ where $[\sigma^{0^{*}}]$ is the set of all binary extensions of $\sigma$.  By compactness, there are $ f_0 , \ldots , f_n \in 2^{\N}$ such that the set of $\cup_{i\leq n}[\bar f_{i} G(f_i)]$ still covers $2^{\N}$.  By \cite{dagsamIII}*{Theorem 3.3}, $\HBU$ is equivalent to the same compactness property for $C$, as follows:
\be\tag{$\HBU_{\c}$}
(\forall G^{2})(\exists  f_{1}, \dots, f_{k} \in C ){(\forall f \in C)}(f\in\cup_{i\leq k} [\overline{f_{i}}G(f_{i})]).
\ee
We now introduce the specification $\SCF(\Theta)$ for a (non-unique) functional $\Theta$ which computes a finite sequence as in $\HBU_{\c}$.  
We refer to such a functional $\Theta$ as a \emph{realiser} for the compactness of Cantor space, and simplify its type to `$3$'.  
\be\tag{$\SCF(\Theta)$}
(\forall G^{2})(\forall f^{1}\leq_{1}1)(\exists g\in \Theta(G))(f\in [\overline{g}G(g)]).
\ee
Clearly, there is no unique such $\Theta$ (just add more binary sequences to $\Theta(G)$) and any functional satisfying the previous specification 
is referred to as a `$\Theta$-functional' or a `special fan functional' or a `realiser for $\HBU$'.
As to their provenance, $\Theta$-functionals were introduced as part of the study of the \emph{Gandy-Hyland functional} in \cite{samGH}*{\S2} via a slightly different definition.  
These definitions are identical up to a term of G\"odel's $T$ of low complexity by \cite{dagsamII}*{Theorem 2.6}.  

\smallskip

Sixth, a number of higher-order axioms are studied in \cite{samph} including:  
\be\tag{$\BOOT$}
(\forall Y^{2})(\exists X\subset \N)\big(\forall n\in \N)(n\in X\asa (\exists f \in \N^{\N})(Y(f, n)=0)\big).
\ee
We only mention that this axiom is equivalent to e.g.\ the monotone convergence theorem for nets indexed by Baire space (see \cite{samph}*{\S3}).  
As it turns out, the coding principle $\open^{+}$ from Section \ref{jackoff} is closely related to $\BOOT$ and fragments, as shown in \cite{samph}.
Historical remarks related to $\BOOT$ are as follows.
\begin{rem}[Historical notes]\label{hist}\rm
%
First of all, $\BOOT$ is definable in Hilbert-Bernays' system $H$ from the \emph{Grundlagen der Mathematik} (\cite{hillebilly2}*{Supplement IV}).  In particular, one uses the functional $\nu$ from \cite{hillebilly2}*{p.\ 479} to define the set $X$ from $\BOOT$.  
In this way, $\BOOT$ and subsystems of second-order arithmetic can be said to `go back' to the \emph{Grundlagen} in equal measure, although such claims may be controversial.  

\smallskip

Secondly, after the completion of \cite{samph}, it was observed by the second author that Feferman's `projection' axiom \textsf{(Proj1)} from \cite{littlefef} is similar to $\BOOT$.  The former is however formulated using sets, which makes it more `explosive' than $\BOOT$ in that full $\Z_{2}$ follows when combined with $(\mu^{2})$, as noted in \cite{littlefef}*{I-12}.  Note that \cite{littlefef} is Paper~154 in Feferman's publication list from \cite{guga}, going back to about 1980.
\end{rem}

\section{Reverse Mathematics and the Axiom of Choice}\label{FRAC}
\subsection{Introduction and basic results}\label{basik}
A number of results in \cite{dagsamV, dagsamVII, samph} exhibit the Pincherle phenomenon from Section \ref{pisec}.
In particular, certain equivalences are established using $\QFAC^{0,1}$, while they (often) cannot be established without $\QFAC^{0,1}$.  
At the same time, a much stronger system \emph{not involving $\QFAC^{0,1}$} proves both members of these equivalences.  
In this section, we show that countable choice can be avoided in favour of $\NCC$ from Section \ref{pisec}.
Unsurprisingly, the proofs become more complex and require greater attention to detail. 
Here is a list of theorems from \cite{dagsamV, dagsamVII, samph} to be treated in the aforementioned way.
\begin{itemize}
\item Pincherle's original theorem for Cantor Space (Section \ref{hijo}).
\item The Heine-Borel theorem for \emph{countable coverings} (Section \ref{hbsec}).
\item The Urysohn lemma and Tietze extension theorem (Section \ref{utsec}).
\item The bootstrap axiom $\BOOT$ and the coding of open sets (Section \ref{jackoff}).
\end{itemize}
We only establish the sufficiency of $\NCC$ for these results, while similar results can be treated in the same way.   

\smallskip

The above results are established based on \cite{kohlenbach2}*{\S3} as follows.  As noted in Section~\ref{prelim1}, $(\exists^{2})$ is equivalent to the existence of a discontinuous function on $\R$. 
Hence, $\neg(\exists^{2})$ is equivalent to the statement \emph{all functions on $\R$ are continuous}.  In the latter case, higher-order statements, like e.g.\ $\PIT_{o}$, often reduce to well-known second-order results.   
Since all systems here are classical, we can therefore invoke the law of excluded middle $(\exists^{2})\vee \neg(\exists^{2})$ and split a given proof in e.g.\ $\RCAo$ or $\RCAo+\WKL$ into two parts: one assuming $\neg(\exists^{2})$ which often reduces to second-order results, and a second part assuming $(\exists^{2})$, where the latter is much stronger than the base theory and $\WKL$.  This `excluded middle trick' was pioneered in \cite{dagsamV}.

\smallskip

For the rest of this section, we discuss some basic results and observations regarding $\NCC$.
First of all, consider the following axiom, called $\Delta$-comprehension, essential for many `lifted' proofs from \cites{samrecount, samFLO2, samph}.
\begin{align}
(\forall Y^{2}, Z^{2})\big[ (\forall n^{0})( (\exists f^{1})&(Y(f, n)=0) \asa (\forall g^{1})(Z(g, n)=0) )\tag{$\DCA$} \\
& \di (\exists X^{1})(\forall n^{0})(n\in X\asa (\exists f^{1})(Y(f, n)=0)\big]\notag
\end{align}
Now, $\DCA$ is mapped to recursive comprehension from $\RCA_{0}$ by $\ECF$, i.e.\ the former axiom is needed to do higher-order RM in a fashion similar to second-order RM.  
We have the following theorem, establishing the basic properties of $\NCC$.
\begin{thm}\label{basicfacts}~
\begin{itemize}
\item The system $\RCAo+\BOOT$ proves $\NCC$.  
\item The system $\RCAo$ proves $\QFAC^{0,1}\di \NCC\di \DCA$.
\end{itemize}
\end{thm}
\begin{proof}
The first item is trivial as $\RCAo$ includes $\QFAC^{0,0}$.  
The first implication in the second item is immediate.  For the second implication in the second item,
consider $Y^{2}, Z^{2}$ that satisfy the antecedent of $\DCA$, i.e.\ 
\[
(\forall n^{0})( (\exists f^{1})(Y(f, n)=0) \asa (\forall g^{1})(Z(g, n)>0) ).
\]
Now apply $\NCC$ to the following (trivial) formula
\[
(\forall n^{0})(\exists m^{0})\big[m=0\di  (\exists f^{1})(Y(f, n)=0)  \wedge (\forall g^{1})(Z(g, n)>0)\di m=0  \big]
\]
to obtain the set required for $\DCA$.  
\end{proof}
We also note that the axiom $\textsf{A}_{0}$ from \cite{samph}*{\S5} trivially implies $\NCC$.  The former axiom is used in \cite{samph} to calibrate theorems based 
on fragments of the \emph{neighbourhood function principle} $\NFP$ (\cite{troeleke1}), a scale finer than (higher-order) comprehension.

\smallskip

Finally, $\NCC$ is also interesting for conceptual reasons: as shown in \cite{dagsamX}*{\S3.2}, $\NCC$ implies the principle $\NBI$, that there is no \emph{bijection} from $[0,1]$ to $ \N$, but $\NCC$ cannot prove $\NIN$, 
that there is no \emph{injection} from $[0,1]$ to $\N$, even when combined with $\Z_{2}^{\omega}$.  
Thus, $\NCC$ is intimately connected to the uncountability of $\R$.

\subsection{Pincherle's theorem}\label{hijo}
In this section, we show that $\NCC$ suffices to obtain the equivalence $\WKL\asa \PIT_{o}$, where the latter is Pincherle's `original' theorem, which 
is mentioned in Section \ref{pisec} and defined as in $\PIT_{o}$ below:
\be\label{LOC}\tag{$\LOC(F, G)$}
 (\forall f , g\in C)\big[ g\in [\overline{f}G(f)] \di F(g)\leq G(f)    \big],
\ee
\be\tag{$\PIT_{o}$}
(\forall F, G:C\di \N)\big[  \LOC(F, G)\di (\exists N\in \N)(\forall g \in C)(F(g)\leq N)\big].
\ee
Note that $\LOC(F, G)$ expresses that $F$ is locally bounded on $2^{\N}$ and $G$ realises this fact.  
As discussed in \cite{dagsamV}, Pincherle explicitly assumes such realisers in \cite{tepelpinch}.
Corollary \ref{waha} deals with $\PIT_{o}$ without such realisers. 
\begin{thm}\label{hijot} 
The system $\RCAo+\NCC$ proves $\WKL\asa \PIT_{o}$.
\end{thm}
\begin{proof} 
The reverse implication is proved in \cite{dagsamV}*{Cor.\ 4.8} over $\RCAo$.  For the forward direction, let $F:2^{\N}\di \N$ be a totally bounded function with realiser $G:2^{\N}\di \N$, i.e.\ we have $(\forall f, g\in 2^{\N})(g \in [\overline{f}G(f)] \di F(g)\leq G(f) )$.  
In case $\neg(\exists^{2})$, $F$ is continuous by \cite{kohlenbach2}*{\S3} and it is well-known that $\WKL$ suffices to prove that $F$ has an upper bound \emph{in this case} (see \cite{kohlenbach4}*{\S4}).  
In case $(\exists^{2})$, suppose $F$ is unbounded on $2^{\N}$, i.e.\ $(\forall n^{0})(\exists f\ \in 2^{\N})(F(f)\geq n\)$.  The following is immediate:
\be\label{worng}
(\forall n^{0})(\exists \sigma^{0^{*}}\leq_{0^{*}}1)\big[|\sigma| = n \wedge (\exists g\in 2^{\N})(F(\sigma * g)\geq n)\big].
\ee
The formula in big square brackets has the right form (modulo coding) to apply $\NCC$.  
Let $H^{0\di 0^{*}}$ be the sequence thus obtained and define $f_{n}:= H(n)*00\dots$.  
Since $(\exists^{2})$ is given, the sequence $f_{n}$ has a convergent subsequence $f_{h(n)}$ with limit $g_{0}$ (see \cite{simpson2}*{III.2}), i.e.\ we have
\be\label{hinge}
 (\forall k^{0})(\exists n^{0})(\forall m^{0}\geq n)(\overline{g_{0}}k=_{0^{*}} \overline{f_{h(m)}}k).
\ee
Now, apply \eqref{hinge} for $k_{0}=G(g_{0})+1$ and obtain the associated $n_{0}$.
For $m_{0}=\max (n_{0}, G(g_{0})+1)$, we then have that $\overline{f_{h(m_{0})}}h(m_{0})*g \in [\overline{g_{0}}G(g_{0})]$ for any $g\in 2^{\N}$ as $h(m_{0})\geq m_{0}\geq G(g_{0})+1$, and hence $F(\overline{f_{h(m_{0})}}h(m_{0})*g)\leq G(g_{0})$ for any $g\in 2^{\N}$ by local boundedness.  
However, the definition of $f_{h(m_{0})}$ implies that there is $g_{1}\in 2^{\N}$ such that $F(\overline{f_{h(m_{0})}}h(m_{0})*g_{1})=F(H(h(m_{0}))*g_{1})\geq h(m_{0})$.  
The assumption $h(m_{0})\geq m_{0}\geq G(g_{0})+1$ thus yields a contradiction. 
\end{proof} 
Finally, let $\PIT_{o}'$ be $\PIT_{o}$ with the antecedent weakened as follows:
\be\label{culk}
(\forall f \in C)(\exists n^{0})(\forall g\in C)\big[ g\in [\overline{f}n] \di F(g)\leq n   \big].
\ee
As expected, \eqref{culk} gives rise to the following corollary.
\begin{cor}\label{waha}
The system $\RCAo+\NCC$ proves $\WKL\asa \PIT_{o}'$.
\end{cor}
\begin{proof}
Replace $G(g_{0})$ with the number $n_{1}^{0}$ obtained for $f=g_{0}$ in \eqref{culk}.
\end{proof}
The previous results should be contrasted with the fact that $\Z_{2}^{\omega}$ cannot prove $\PIT_{o}$, while $\PIT_{o}$ is provable in $\Z_{2}^{\Omega}$ (and hence $\ZF$).

\subsection{Closed and open sets}
We study theorems named after Tietze and Urysohn (Section \ref{utsec}) and Heine and Borel (Section \ref{hbsec}), formulated using \emph{higher-order} open and closed sets.  
The latter notion is introduced in Section \ref{copen}, along with more details.  In each case, we show that $\NCC$ can replace the use of $\QFAC^{0,1}$.

\subsubsection{Introduction}\label{copen}
In this section, we study theorems from \cite{dagsamVII} that exhibit the Pincherle phenomenon.  
In particular, we show that $\QFAC^{0,1}$ is not necessary, but that $\NCC$ suffices in these results. 
These theorems pertain to open and closed sets given by characteristic functions, defined as follows.   
\bdefi[Open sets in $\RCAo$ from \cite{dagsamVII}]\label{openset}
We let $Y: \R \di \R$ represent open subsets of $\R$ as follows: we write `$x \in Y$' for `$|Y(x)|>_{\R}0$' and call a set $Y\subseteq \R$ Ôopen' if for every $x \in Y$, there is an open ball $B(x, r) \subset Y$ with $r^{0}>0$.  
A set $Y$ is called `closed' if the complement, denoted $Y^{c}=\{x\in \R: x\not \in Y \}$, is open. 
\edefi
We have argued in \cite{dagsamVII} that this definition remains close to the `$\Sigma_{1}^{0}$-definition' of open set used in RM. 
In the case of sequential compactness, Definition \ref{openset} yields the known results involving $\ACA_{0}$, while \emph{countable} open-cover compactness
already gives rise to the Pincherle phenomenon, as sketched in Section \ref{hbsec}. 

\smallskip

For the rest of this section, `open' and `closed' refer to Definition \ref{openset}, while `RM-open' and `RM-closed' refer to the usual RM-definition from \cite{simpson2}*{II.4}. 

\subsubsection{Heine-Borel theorem}\label{hbsec}
We now study the Heine-Borel theorem for countable covers of closed sets as in Definition \ref{openset}. 
Note that the associated theorem for RM-codes is equivalent to $\WKL$ by \cite{brownphd}*{Lemma 3.13}.  
\bdefi[$\HBC$]
Let $C\subseteq [0,1]$ be a closed set and let $a_{n}, b_{n}$ be sequences of reals such that $C\subseteq \cup_{n\in \N}(a_{n}, b_{n})$.  Then there is $n_{0}$ such that $C\subseteq \cup_{n\leq n_{0}}(a_{n}, b_{n})$.
\edefi
\noindent
It is shown in \cite{dagsamVII} that $\HBC$ has the following properties.  
\begin{itemize}
\item The system $\RCAo+\QFAC^{0,1}$ proves $\WKL\asa \HBC$.
\item The system $\Z_{2}^{\omega}$ cannot prove $\HBC$, while $\Z_{2}^{\Omega}$ (and $\RCAo+\HBU$) can.
\end{itemize}
By the these items, $\HBC$ clearly exhibits the Pincherle phenomenon.  
Note that by the second item, $\HBC$ is provable without countable choice and has weak first-order strength.  
We let $\HBC_{\RM}$ be $\HBC$ with $C\subseteq [0,1]$ represented by RM-codes.  

\smallskip

We now prove the following theorem. 
\begin{thm}\label{dich}
The system $\RCAo+\NCC$ proves $\WKL\asa \HBC$.
\end{thm}
\begin{proof}
The reversal can be found in \cite{dagsamVII}*{Cor.\ 3.4} over $\RCAo$.  
It also follows from taking $C=[0,1]$ in $\HBC$ and applying \cite{simpson2}*{IV.1.2}.
For the forward direction, in case $\neg(\exists^{2})$, all functions on $\R$ are continuous by \cite{kohlenbach2}*{\S3}.  Following the results in \cite{kohlenbach4}*{\S4}, 
continuous functions have an RM-code on $[0,1]$ given $\WKL$, i.e.\ 
our definition of open set reduces to an $\L_{2}$-formula in $\Sigma_{1}^{0}$, which (equivalently) defines a code for an open set by \cite{simpson2}*{II.5.7}.  In this way, $\HBC$ is merely $\HBC_{\RM}$, which follows
from $\WKL$ by \cite{brownphd}*{Lemma 3.13}.    
In case $(\exists^{2})$, let $C\subseteq [0,1]$ be a closed set and let $a_{n}, b_{n}$ be as in $\HBC$.  
If there is no finite sub-cover, then we also have that
\be\label{slightvar}
(\forall m^{0})(\exists q\in \Q)(\exists x\in C)\big[ [x](m)=q\wedge    x\not \in  \cup_{n\leq m}(a_{n}, b_{n})\big].  
\ee
Apply $\NCC$ and $(\exists^{2})$ to \eqref{slightvar}, yielding a sequence $(q_{n})_{n\in \N}$ of rationals in $C$ with this property.
Since $(\exists^{2})\di \ACA_{0}$, any sequence in $[0,1]$ has a convergent sub-sequence \cite{simpson2}*{III.2}.  Let $h:\N\di \N$ be such that $y_{n}:= q_{h(n)}$ converges to $y\in [0,1]$.

\smallskip

If $y\not\in C$, then there is $N^{0}$ such that $B(y, \frac{1}{2^{N}})\subset C^{c}$, as the complement of $C$ is open by definition. 
However, $y_{n}$ is eventually in $B(y, \frac{1}{2^{N}})$ by definition, a contradiction.   
Note that $y_{n}$ may not be in $C$, but elements of $C$ are arbitrarily close to $y_{n}$ for large enough $n$ by \eqref{slightvar}.  

\smallskip

Hence, we may assume $\lim_{n\di \infty}y_{n}=y\in C$.  However, then $y\in (a_{k}, b_{k})$ for some $k$, and $y_{n}$ is eventually in this interval.  In the same way as in the previous case, this yields a contradiction.
The law of excluded middle now finishes the proof.
\end{proof}
\noindent
As shown in \cite{dagsamVII}*{\S3}, the following theorems imply $\HBC$ over $\RCAo$:
\begin{enumerate}
 \renewcommand{\theenumi}{\alph{enumi}}
\item PincherleÕs theorem for $[0,1]$\textup{:} a locally bounded function on $[0,1]$ is bounded.\label{mikkel}  
\item If $F^{2}$ is continuous on a closed set $D\subset 2^{\N}$, it is bounded on $D$.\label{bounk}
\item If $F^{2}$ is continuous on a closed set $D \subset 2^{\N}$, it is uniformly cont.\ on $D$.\label{bounk2}
\item If $F$ is continuous on a closed set $D\subset [0,1]$, it is bounded on $D$.\label{bounk0}
\item If $F$ is continuous on a closed set $D \subset [0,1]$, it is uniformly cont.\ on $D$.\label{bounk3}
\item If $F$ is continuous on a closed set $D\subset [0,1]$, then for every $\eps>0$ there is a polynomial $p(x)$ such that $|p(x)-F(x)|<\eps$ for all $x\in D$. \label{bounk6}
\end{enumerate}
In the same way as above, one obtains an equivalence between these theorems and $\WKL_{0}$, using $\NCC$ instead of $\QFAC^{0,1}$.

\smallskip

We finish this section with a remark on the Baire category theorem.
\begin{rem}\label{cdki}\rm
The Baire category theorem for open sets as in Definition \ref{openset} is studied in \cite{dagsamVII}*{\S6}. 
Similar to e.g.\ $\HBC$, the Baire category theorem 
exhibits (part of) the Pincherle phenomenon.  The associated proofs for the latter theorem are however 
very different from all other proofs.  Similarly, $\NCC$ does not seem to suffice to prove the Baire category theorem and 
the following one does.  
\bdefi[$\MCC$]
For $Y^{2}$ and $A(n, m)\equiv (\forall g\in 2^{\N})(Y(g, m, n)=0)$:
\[
(\forall n^{0})(\exists m^{0})A(n,m)\di  (\exists h^{1})(\forall n^{0})A(n,h(n)). 
\]
\edefi
We have not found any use for $\MCC$ besides, but it shall be seen to yield 
the same class as realisers as $\NCC$ in Section \ref{cac}.
\end{rem}

\subsubsection{Urysohn's lemma and Tietze's theorem}\label{utsec}
We study the equivalence between the Urysohn lemma $(\URY)$ and the Tietze extension theorem $(\TIET)$, formulated using open sets as in Definition \ref{openset}.  
In particular, this equivalence is proved in \cite{dagsamVII}*{\S5} using $\QFAC^{0,1}$ and we now show that $\NCC$ suffices. 

\smallskip

We first consider the following necessary definitions.  
\bdefi[$\URY$]
For closed disjoint sets $C_{0}, C_{1}\subseteq \R$, there is a continuous function $g:\R\di [0,1]$ such that $x\in C_{i}\asa g(x)=i$ for any $x\in \R$ and $i\in \{0,1\}$.
\edefi
\bdefi[$\TIET$]
For $f:\R\di \R$ continuous on the closed $D\subset[0,1]$, there is $g:\R\di \R$, continuous on $[0,1]$ such that $f(x)=_{\R}g(x)$ for $x\in D$.
\edefi
Secondly, $\URY\asa \TIET$ is proved in \cite{dagsamVII}*{\S5} using $\QFAC^{0,1}$ \emph{and} $\cont$, where the latter is the statement that every continuous $Y:\R\di \R$ has an RM-code, as studied in \cite{kohlenbach4}*{\S4} for Baire space.  
Note that the $\ECF$-interpretation of $\cont$ is a tautology.
We have the following nice equivalence. 
\begin{thm}\label{tieten}
The system $\RCAo+\NCC+\cont$ proves $\TIET\asa \URY$.  
\end{thm}
\begin{proof}
The implication $\URY\di \TIET$ is proved in \cite{dagsamVII}*{\S5} over $\RCAo$.

\smallskip

For $\TIET\di \URY$, in case $\neg(\exists^{2})$, all functions on $\R$ are continuous by \cite{kohlenbach2}*{\S3} and open sets reduce to RM-codes via $\cont$; the usual proof of $\URY$ from \cite{simpson2}*{II.7} then goes through. 
In case $(\exists^{2})$, let $C_{i}$ be as in $\URY$ for $i=0,1$ and define $f$ on $C_{2}:= C_{0}\cup C_{1} $ as follows: $f(x)=0$ if $x\in C_{0}$ and $1$ otherwise.  
\emph{If} $f$ is continuous on $C_{2}$, \emph{then} its extension $g$ provided by $\TIET$ is as required for $\URY$.  To show that 
$f$ is continuous on $C_{2}$, we prove that 
\be\label{sepa}\textstyle
(\forall N^{0})(\exists n^{0})(\forall x\in C_{0}, y\in C_{1})(x, y\in [-N, N]\di |x-y|\geq  \frac{1}{2^{n}}). 
\ee
If \eqref{sepa} is false, there is $N\in \N$ such that for $n\in \N$, there are $q, r\in \Q$ such that:
\[\textstyle
(\exists x\in C_{0}, y\in C_{1})( [x](n+1)=q \wedge [y](n+1)=r  \wedge x, y\in [-N, N]\wedge |x-y|<  \frac{1}{2^{n+1}}). 
\]
Applying $\NCC$ yields sequences $(q_{n})_{n\in\N}$, $ (r_{n})_{n\in \N}$ in $[-N-1,N+1]$
such that for all $n^{0}$, there are $x\in C_{0}, y\in C_{1}$ such that
\be\label{kringel}\textstyle
 [x](n+1)=q_{n} \wedge [y](n+1)=r_{n} \wedge x, y\in [-N, N]\wedge |x-y|<  \frac{1}{2^{n+1}}.
\ee
As these sequences are bounded, there are $x_{0},y_{0}\in [-N, N]$ such that $q_{h_{0}(n)}\di x_{0}$ and $r_{h_{1}(n)}\di y_{0}$ for subsequences provided by $h_{0}, h_{1}:\N\di \N$.  
Since $C_{0}$ is closed, we have the following: if $x_{0}\not \in C_{0}$, then there is $r>0$ such that $B(x_{0}, r)\cap C_{0}=\emptyset$.  This however contradicts the convergence $q_{h_{0}(n)}\di x_{0}$ and \eqref{kringel}.  
Hence $x_{0}\in C_{0}$ and $y_{0}\in C_{1}$ in the same way. 
Now note that  $(\forall n^{0})(|r_{n}-q_{n}|<\frac{1}{2^{n}})$ by \eqref{kringel}, which implies that $x_{0}=_{\R}y_{0}$, a contradiction since $C_{0}\cap C_{1}=\emptyset$. 
Finally, since \eqref{sepa} provides a positive `distance' between $C_{0}$ and $C_{1}$ in every interval $[-N,N]$, we can always chose a small enough neighbourhood to exclude points from one of the parts of $C_{2}$, thus guaranteeing continuity for $f$ everywhere on $C_{2}$. 
\end{proof}
Finally, we point out that while Definition \ref{openset} gives rise to interesting results in \cite{dagsamVII}, we could not obtain (all) the expected RM-equivalences \emph{try as we might}.
A better definition of open set, namely Definition \ref{opensset}, that does yield the expected RM-equivalences was introduced in \cite{samph}.  We now study this `better' definition.

\subsection{Bootstrap axioms}\label{jackoff}
We study equivalences from \cite{samph} involving the `bootstrap' axiom $\BOOT$ and show that the use of $\QFAC^{0,1}$ can be replaced with $\NCC$.

\smallskip

First of all, $[\BOOT+\ACA_{0}]\asa \open^{+}$ was proved using $\QFAC^{0,1}$ in \cite{samph}*{\S4.2}.
The `coding principle' $\open^{+}$ connects open sets as in RM, given by countable unions, and open sets given by \emph{uncountable} unions. 
In this section, `open' and `closed' and  refers to the below definition, while `RM-open' refers to the well-known RM-definition from \cite{simpson2}*{II.5} involving countable unions of basic open balls.
\bdefi[Open sets in $\RCAo$ from \cite{samph}]\label{opensset}
An open set $O$ in $\R$ is represented by a functional $\psi:\R\di \R^{2}$.  We write `$x\in O$' for $(\exists y\in \R)(x\in I_{y}^{\psi})$, where $I_{y}^{\psi}$ is the open interval $\big(\psi(y)(1), \psi(y)(1)+|\psi(y)(2)|\big)$ in case the end-points are different, and $\emptyset$ otherwise.  We write $O=\cup_{y\in \R}I_{y}^{\psi}$ to emphasise the connection to uncountable unions.  
A closed set is represented by the complement of an open set.       
\edefi
Intuitively, open sets are given by \emph{uncountable} unions $\cup_{y\in \R}I_{y}^{\psi}$, just like RM-open sets are given by countable such unions.  
Hence, our notion of open set reduces to the notion RM-open set when applying $\ECF$ or when all functions on $\R$ are continuous.  
Moreover, writing down the definition of elementhood in an RM-open set, one observes that such sets are also open (in our sense). 
Finally, closed sets are readily seen to be sequentially closed, and the same for nets instead of sequences.  

\smallskip

The following `coding principle' turns out to have nice properties.  Note that $\open$, a weaker version of $\open^{+}$, was introduced and studied in \cite{dagsamVII}.
We fix an enumeration of all basic open balls $B(q_{n}, r_{n})\subset \R$ for rational $q_{n}, r_{n}$ with $r_{n}>_{\Q}0$.
\bdefi[$\open^{+}$]
For every open set  $Z\subseteq \R$, there is $X\subseteq \N$ such that $(\forall n\in \N)(n\in X\asa B(q_{n}, r_{n})\subseteq Z)$.
\edefi
Note that given the set $X$ from $\open^{+}$, we can write $Z=\cup_{n\in X}B(q_{n}, r_{n})$ as expected.
We now have the following equivalence. 
\begin{thm}
The system $\RCAo+\NCC$ proves $\BOOT\asa[ \open^{+} +\ACA_{0}]$.
\end{thm}
\begin{proof}
The implication $[\ACA_{0}+\open^{+}]\di \BOOT$ over $\RCAo$ is immediate from \cite{samph}*{Theorem~4.4}.
We now prove the `crux' implication $\BOOT\di \open^{+}$ using $\NCC$.  In case $\neg(\exists^{2})$, all functionals on $\R$ or $\N^{\N}$ are continuous by \cite{kohlenbach2}*{\S3}.
Thus, an open set $\cup_{y\in \R}I_{y}^{\psi}$ reduces to the \emph{countable} union $\cup_{q\in \Q}I_{q}^{\psi}$, yielding $\open^{+}$ in this case.     
In case $(\exists^{2})$, let $O$ be an open set given by $\psi:\R\di \R^{2}$ as in Definition~\ref{opensset}.  
Now use $\BOOT$ and $(\exists^{2})$ to define the following set $X\subset \N\times \Q$:
\be\label{flim}\textstyle
(\forall n\in \N, q\in \Q)\big( (n,q)\in X\asa (\exists y\in \R)\big(B(q, \frac{1}{2^{n}})\subset I_{y}^{\psi} \big)  \big).
\ee
Trivially, for the set $X$ from \eqref{flim}, we have for all $n\in \N, q\in \Q$ that:
\begin{align}\textstyle
 (n,q)\in X\di  (\exists m\in \N, r\in \Q)\big(&\textstyle B(q, \frac{1}{2^{n}})\subseteq B(r, \frac{1}{2^{m}}) \notag\\
 &\textstyle\wedge (\exists y\in \R)( B(r, \frac{1}{2^{m}})\subseteq   I_{y}^{\psi}) \big). \label{flimq}
\end{align}
Apply $\NCC$ to the implication in \eqref{flimq} to obtain $\Phi$ such that for all $n\in \N, q\in \Q$:
\begin{align}\textstyle
 (n,q)\in X\di \big(B(q, \frac{1}{2^{n}})&\textstyle\subseteq B(\Phi(n, q)(1), \frac{1}{2^{\Phi(n,q)(2)}})\notag\\
& \textstyle\wedge (\exists y\in \R)( B(\Phi(n,q)(1), \frac{1}{2^{\Phi(n,q)(2)}})\subseteq   I_{y}^{\psi}) \big). \label{flim2}
\end{align}
Now consider the following formula defined in terms of the above $X$ and $\Phi$.
\begin{gather}\notag\textstyle
x\in O\asa (\exists n\in \N, q\in \Q)( (n, q)\in X\wedge x \in B(\Phi(n, q)(1), \frac{1}{2^{\Phi(n,q)(2)}})\\\wedge \textstyle\underline{(\exists y\in \R)( B(\Phi(n,q)(1), \frac{1}{2^{\Phi(n,q)(2)}})\subseteq   I_{y}^{\psi})} \big). \label{keind}
\end{gather}
Note that $\BOOT$ provides a set $Y$ such that $(q, n)\in Y$ if and only $q, n$ satisfy the underlined formula in \eqref{keind}.  Thus, the right-hand side of \eqref{keind} is decidable given $(\exists^{2})$.
The formula \eqref{keind} provides a representation of $O$ as a countable union of open balls, and of course gives rise to $\open^{+}$.  What is left is to prove \eqref{keind}.   

\smallskip

For the reverse implication in \eqref{keind}, $x\in O$ follows by definition from the right-hand side of \eqref{keind}. 
For the forward implication, $x_{0}\in O$ implies $B(x_{0},\frac{1}{2^{n_{0}}} )\subset I_{y_{0}}^{\psi}$ for some $y_{0}\in \R$ and $n_{0}\in \N$ by definition.
For $n_{1}$ large enough, the rational $q_{0}:=[x_{0}](n_{1})$ is inside $B(x_{0},\frac{1}{2^{n_{0}+1}} )$.   
Hence, $(q_{0}, n_{0}+1)\in X$ by \eqref{flim} for $y=y_{0}$.  Applying \eqref{flim2} then yields 
\begin{align}\label{helf}\textstyle
B(q_{0}, \frac{1}{2^{n_{0}+1}})\subseteq ~&\textstyle B(\Phi(n_{0}+1, q_{0})(1), \frac{1}{2^{\Phi(n_{0}+1,q_{0})(2)}})\\\textstyle
&\textstyle\wedge (\exists y\in \R)( B(\Phi(n_{0}+1,q_{0})(1), \frac{1}{2^{\Phi(n_{0}+1,q_{0})(2)}})\subseteq   I_{y}^{\psi}) \big)\notag
\end{align}
By assumption, we also have $x_{0}\in B(q_{0}, \frac{1}{2^{n_{0}+1}})$, and the right-hand side of \eqref{keind} thus follows from \eqref{helf}, and we are done.
\end{proof}
The previous theorem has numerous implications.  For instance, it is proved in \cite{samph}*{\S4} that $[\ACA_{0}+\CBT]\asa [\FIVE+\BOOT]$ over $\RCAo+\QFAC^{0,1}$, 
where $\CBT$ is the Cantor-Bendixson theorem, defined as follows.
\begin{princ}[$\CBT$]
For any closed set $C\subseteq [0,1]$, there exist $P, S\subset C$ such that $C=P\cup S$, $P$ is perfect and closed, and $S^{0\di 1}$ is a sequence of reals. 
\end{princ}
It goes without saying that the above equivalence involving $\CBT$ can be obtained using only $\NCC$ instead.  The same holds for the perfect set theorem and theorems pertaining to separably closed sets from \cite{samph}*{\S4}.


\section{Computability theory and the Axiom of Choice}\label{cac}
\subsection{Introduction}
We study the computational properties of $\NCC$ and related principles.  
To this end, we first introduce the concept of `realiser for $\NCC$'.  
\bdefi[$\NCC(\zeta)$]\label{zzita}
For $Y^{2}$ and $A(n, m)\equiv (\exists f\in 2^{\N})(Y(f, m, n)=0)$:
\[
(\forall n^{0})(\exists m^{0})A(n,m)\di  (\forall n^{0})A(n,\zeta(Y)(n)). 
\]
We refer to $\zeta^{2\di 1}$ satisfying $\NCC(\zeta)$ as a `realiser for $\NCC$' or `$\zeta$-functional'.
\edefi
Note that $\zeta$-functionals as in the previous definition are trivially computable in $\exists^{3}$ via a term of G\"odel's $T$ of very low complexity.  
We are also interested in \emph{weak} realisers for $\NCC$, which are $\zeta_{\w}^{2\di 1}$ such that $(\forall n^{0})(\exists m^{0}\leq \zeta_{\w}(Y)(n))A(n,m)$ in the above specification. 
Thus, $\zeta_{\w}$-functionals only provide an upper bound for the choice function in $\NCC$, while $\zeta$-functionals provide such a function, as is clear from Definition \ref{zzitaw}.  This modification has been discussed in Section \ref{pisec}.

\smallskip

We are also interested in the following related specification for $\vartheta$-functionals, 
which are realisers for $\MCC$ as in Remark \ref{cdki}.
\bdefi[$\MCC(\vartheta)$]
For $Y^{2}$ and $A(n, m)\equiv (\forall g\in 2^{\N})(Y(g, m, n)=0)$:
\[
(\forall n^{0})(\exists m^{0})A(n,m)\di  (\forall n^{0})A(n,\vartheta(Y)(n)). 
\]
\edefi
As noted in Remark \ref{cdki}, the Pincherle phenomenon also pops up when studying the 
Baire category theorem for open sets given by characteristic functions.  However, the associated
proofs are \emph{completely different} from those for the Heine-Borel or Pincherle theorem.  
Similarly, $\vartheta$-functionals give rise to a realiser for the Baire category theorem, while the former
seem fundamentally different from $\zeta$-functionals.  

\smallskip

In Section \ref{cmc}, we show that while $\NCC$ and $\MCC$ are rather weak (from a first-order strength perspective), its \emph{total} realisers are quite strong (from a computational perspective) in that they are exactly $\exists^{3}$.  
Interestingly, this result makes use of a relatively strong fragment of the \emph{axiom of extensionality}; the latter is included in $\RCAo$ as \eqref{EXT} for all finite types. 

\smallskip

In Section \ref{cmc2}, we study \emph{partial} realisers of $\NCC$; one expects those to be weaker than their total counterparts.
We show that such partial realisers can perform the computational task (B) from Section \ref{pisec}; we also \emph{conjecture} that such partial realisers cannot perform the seemingly stronger task (A).  
Since we do not have a proof of this conjecture, we will tackle a weaker problem, namely to find a partial realiser of $\NCC$ that does not compute $\exists^{3}$. 
As noted above, a useful concept is that of a \emph{countably based} functional as in Definition~\ref{frag}.  Indeed, since $\exists^{3}$ is not countably based, a \emph{countably based} partial realiser for $\NCC$ cannot compute $\exists^{3}$.  
In other words, such a partial realiser would be exactly what we want.  
This construct does exist, but is rather elusive: by Theorem \ref{thm.CH} the existence of {countably based} partial realiser for $\NCC$ is equivalent to the Continuum Hypothesis ($\CH$ for short).

\subsection{Total realisers}\label{cmc}
We show that Kleene's $\exists^{3}$ and various \emph{total} realisers for $\NCC$ are one and the same thing, even in weak systems.  

\subsubsection{The power of total realisers for $\NCC$}\label{totsect}
In this section, we show that $\zeta$-functionals compute $\exists^{3}$ and vice versa. 
We also obtain associated equivalences over the base theory $\RCAo$. 
To this end, we first establish the following two lemmas. 
\begin{lemma}
Any $\zeta$-functional computes $\kappa_{0}$ via a term of G\"odel's $T$.
The system $\RCAo$ proves $(\exists \zeta)\NCC(\zeta)\di (\kappa_{0}^{3})$.
\end{lemma}
\begin{proof}
Fix some functional $Y^{2}$ and define the following sequence:
\be\label{klop}
Y_{k}(n, m, f):=
\begin{cases}
0 &Y(f)=0\wedge m=k \\
1 & \textup{ otherwise }
\end{cases}.
\ee
Let $\zeta$ be as in $\NCC(\zeta)$ and consider the following formula.
\be\label{cruxs}
(\exists f\in 2^{\N})(Y(f)=0)\asa \zeta(Y_{0})(0)\ne_{0} \zeta(Y_{1})(0).
\ee
Since the right-hand side of \eqref{cruxs} is decidable, this formula gives rise to $\kappa_{0}^{3}$, as required by the lemma.  
To prove \eqref{cruxs}, note that $(\exists f\in 2^{\N})(Y(f)=0)$ implies $\zeta(Y_{0})(0)=0$ and $\zeta(Y_{1})(0) = 1$ by the definition in \eqref{klop}.
For the remaining implication, $(\forall f\in 2^{\N})(Y(f)>0)$ implies $Y_{0}=_{2}Y_{1}=_{2}1$, i.e.\ the latter functionals are constant $1$.  
The axiom of extensionality $(\textsf{E})_{2,0}$ then yields $\zeta(Y_{0})(0)=_{0}\zeta(Y_{1})(0)$, as required. 
The equivalence \eqref{cruxs} now finishes the proof. 
\end{proof}
We note that the above proof fails if the $\zeta$-functional at hand is not total, while we can prove that there is a \emph{partial} $\zeta$-functional computable in $\kappa^3_{0}$.
We also point out that the axiom of extensionality (for a relatively high type) is used in an essential way in the reverse implication in \eqref{cruxs}.
\begin{lemma}\label{kiolp}
Any $\zeta$-functional computes $\exists^{2}$ via a term of G\"odel's $T$.
The system $\RCAo$ proves $(\exists \zeta)\NCC(\zeta)\di (\exists^{2})$.
\end{lemma}
\begin{proof}
Fix $Y^{2}$ and $\zeta$ as in $\NCC(\zeta)$.
Using dummy variables and $\zeta(Y)(0)$, we can define $\zeta_0^{2\di 0}$ such that whenever $(\exists m^{0},  \exists f \in 2^{\N})(Y(f, m)=0) $ then $\zeta_0(Y) = m_{0}$ such that $(\exists f\in 2^{\N})(Y(f, m_{0})=0)$.
Now fix $g^{1}$ and define $Y$ as follows
\[
Y(n,f) = 
\begin{cases} 0 & \textup{if $ f(n) = 0$} \\
 1 & \textup{otherwise} \end{cases}
 \]
Then $(\exists n^{0}) (g(n) = 0) \leftrightarrow g(\zeta_0(Y)) = 0$, and we are done. 
\end{proof}
We again point out that the axiom of extensionality (for a relatively high type) is used in an essential way in the final equivalence in the proof.  
To the best of our knowledge, the axiom of extensionality has not been used in higher-order RM beyond $(\textsf{E})_{1, 0}$ in formalising \emph{Grilliot's trick} in $\RCAo$ (see \cite{kooltje, kohlenbach2}).

\smallskip

We now have the following main theorem of this section. 
\begin{thm}\label{maink}
The functional $\exists^{3}$ computes a $\zeta$-functional via a term of G\"odel's $T$, and vice versa.
The system $\RCAo$ proves $(\exists \zeta)\NCC(\zeta)\asa (\exists^{3})$.
\end{thm}
\begin{proof}
That $\exists^{3}$ computes a $\zeta$-functional is immediate from the fact that the former computes Feferman's $\mu^{2}$.
The reverse computational direction is similarly immediate in light of the above lemmas.  
For the forward implication, the splitting $(\exists^{3})\asa [(\kappa_{0}^{3})\asa (\exists^{2})]$ can be found in \cite{dagsam}*{\S6}, going back to Kohlenbach.  
Combining the two above lemmas yields the forward implication, while the reverse one is immediate in light of $(\exists^{2})\asa (\mu^{2})$ over $\RCAo$ (see \cite{kooltje}).
\end{proof}
The following corollary is immediate by the theorem, while the second corollary follows \emph{mutatis mutandis}.
\begin{cor}
In the specification $\NCC(\zeta)$, we may assume that $\zeta(Y)(n)$ provides the \emph{least} witness to $m$. 
\end{cor}
\begin{cor}
The functional $\exists^{3}$ is computable from a $\vartheta$-functional via a term of G\"odel's $T$, and vice versa.
The system $\RCAo$ proves $(\exists \vartheta)\MCC(\vartheta)\asa (\exists^{3})$.
\end{cor}
In light of the above, realisers for $\NCC$ and $\MCC$ are (too) strong and we shall study weaker objects in the next section.
Nonetheless, it is interesting that we have obtained a very different equivalent formulation for $(\exists^{3})$ based on a fragment $\AC$, namely $\NCC$.
It is also interesting that we seem to need a relatively strong fragment of the axiom of extensionality.  Similar to \cite{kooltje}, it is a natural question whether
the above equivalences go through \emph{without} the latter axiom. 

\subsubsection{The power of weak total realisers for $\NCC$}
Similar to the previous section, we study \emph{weak} realisers for $\NCC$ as in Defintion \ref{zzitaw} below. 
As noted in Section~\ref{pisec}, weakening $\NCC$ as in the latter definition means that the proofs in Section \ref{FRAC} do not (seem to) go through. 
Nonetheless, we show that these weak realisers for $\NCC$ still compute $\exists^{3}$, and vice versa.
\bdefi[$\NCC_{\w}(\zeta_{\w})$]\label{zzitaw}
For $Y^{2}$ and $A(n, m)\equiv (\exists f\in 2^{\N})(Y(f, m, n)=0)$:
\[
(\forall n^{0})(\exists m^{0})A(n,m)\di  (\forall n^{0})(\exists m\leq \zeta_{\w}(n))A(n,m). 
\]
\edefi
The following lemma is proved in the same way as for Lemma \ref{kiolp}.
\begin{lemma}\label{frollic}
Any $\zeta_{\w}$-functional computes $\exists^{2}$ via a term of G\"odel's $T$.
The system $\RCAo$ proves $(\exists \zeta_{\w})\NCC_{\w}(\zeta_{\w})\di (\exists^{2})$.
\end{lemma}
\begin{proof}
Use the same functional $Y$ as in the proof of Lemma \ref{kiolp}, observing that $(\exists n^{0})(g(n) = 0) \leftrightarrow( \exists n \leq \zeta_{\w}(Y))(g(n) = 0)$. 
\end{proof}
We also have the following (surprising) result showing that even weak realisers for $\NCC$ are in fact strong.
\begin{lemma}
Any $\zeta_{\w}$-functional computes $\kappa_{0}$ via a term of G\"odel's $T$.
The system $\RCAo$ proves $(\exists \zeta_{\w})\NCC_{\w}(\zeta_{\w})\di (\kappa_{0}^{3})$.
\end{lemma}
\begin{proof}
In the same way as in the proof of Lemma \ref{kiolp}, use $\zeta_{\w}$ to define $\zeta_0^{2\di 1}$ such that if $(\exists n^{0}) (\exists f \in 2^{\N}) (Y(n,f)=0)$ then $( \exists n \leq \zeta_0(Y)) (\exists f\in 2^{\N}) (Y(n,f)=0)$.
Fix some $Z^{2}$ and define two functionals $Y_{i}$ for $i=0,1$ as follows:
\[
Y_{0}(n, f):=
\begin{cases}
0 & Z(f)=0 
\\
1 & \textup{ otherwise}
\end{cases}
~
Y_{1}(n, f):=
\begin{cases}
0 & Z(f)=0 \wedge n >\zeta_{0}(Y_{0})\\
1 & \textup{ otherwise}
\end{cases}.
\]
Then $\kappa_{0}^{3}$ is obtained by the previous lemma and the following equivalence:
\be\label{lepanu}
(\exists f\in 2^{\N})(Z(f)=0)\asa \zeta_{0}(Y_{0})\ne_{1}\zeta_{0}(Y_{1}).  
\ee
For the forward direction in \eqref{lepanu}, note that $(\exists f\in 2^{\N})(Z(f)=0)$ implies that $\zeta_{0}(Y_{1}))(0)>_{0}\zeta_{0}(Y_{0})(0)$ by the definition of $Y_{1}$.  
For the reverse direction in \eqref{lepanu}, assuming $(\forall f\in 2^{\N})(Z(f)>0)$ yields $Y_{0}=_{2}Y_{1}=_{2}1$, and the axiom of extensionality $\textsf{(E)}_{2,1}$ yields $\zeta_{0}(Y_{0})=_{1}\zeta_{0}(Y_{1})$, as required.
\end{proof}
We now easily obtain the other main result of this section. 
\begin{thm}\label{maink2}
The functional $\exists^{3}$ computes a $\zeta_{\w}$-functional via a term of G\"odel's $T$, and vice versa. 
The system $\RCAo$ proves $(\exists \zeta_{\w})\NCC(\zeta_{\w})\asa (\exists^{3})$.
\end{thm}
What makes the results in this section interesting is that realisers for $\NCC$ grew out of a principle that seemed natural and weak, and these functionals then turned out to be strong. 
This illustrates the power of assuming that realisers are total, and supports our view that the partial $\zeta$-functionals reflect in a more natural way the principle $\NCC$ that is meant to replace countable choice.
Hence, we shall study partial realisers for $\NCC$ in Section \ref{cmc2}.  
\subsection{Partial realisers}\label{cmc2}
In this section, we study partial realisers for $\NCC$ and show that they are weaker and have more interesting computational properties than total realisers for $\NCC$. 

\smallskip

In Section \ref{incel}, we connect these realisers to the computational study of compactness as in items (A) and (B) from Section \ref{pisec}.  
In Section \ref{amehoela}, we show that the existence of \emph{countably based} partial realisers for $\NCC$ is equivalent to the Continuum Hypothesis. 
We further provide a foundational discussion of partial versus total functionals in Section \ref{labbe}. Finally, in Section \ref{optional} we discuss the role of Kleene computability in our endeavour and a possible weaker alternative.

\subsubsection{The power of partial realisers for $\NCC$}\label{incel} 
We introduce the notion of `partial realiser for $\NCC$' and prove some basic properties. 

\smallskip

First of all, the following definition is as expected. 
\bdefi[Partial realisers for $\NCC$]\rm~
\begin{itemize}
\item[(a)] A \emph{partial ${\rm \NCC}$-realiser} is a partial functional $\zeta_{\p}$ taking objects $Y$ of type  $ (\N^2 \times 2^{\N}) \rightarrow \N$ as arguments such that \textbf{if}
\[
(\forall n^{0}) (\exists m^{0}) (\exists f \in 2^{\N})(Y(n,m,f)= 0)
\] 
\textbf{then} $\zeta_\p(Y) = g$ is a choice function satisfying
\[
(\forall n^{0})( \exists f \in 2^{\N })(Y(n,g(n),f))= 0).
\]
\item[(b)] A \emph{weak partial ${\rm \NCC}$-realiser} is a partial functional $\zeta_{\p_0}$ taking objects $Y$ of type $( \N \times 2^{\N}) \rightarrow \N$ as arguments, such that \textbf{if} 
$(\exists m^{0})( \exists f \in 2^{\N})(Y(m,f) = 0)$ \textbf{then} $\zeta_{\p_0}(Y)$ terminates and yields an $m$ such that $(\exists f \in 2^{\N} )(Y(m,f)= 0)$.
\end{itemize}
\edefi
While seemingly different, items (a) and (b) yield the same computational class. 
\begin{lemma}\label{lem5} 
The classes of partial ${\rm \NCC}$-realisers and weak partial ${\rm \NCC}$-realisers are computationally equivalent.
\end{lemma}
\begin{proof}
Clearly a partial ${\rm \NCC}$-realiser computes a weak one: to compute $\zeta_{\p_0}(Y)$, one computes $\zeta_\p(\lambda (n,m,f).Y(m,f))(0)$.
Given $\zeta_{\p_0}$ we can compute $\zeta_\p(Y)(n) = \zeta_{\p_0}(\lambda (m,f) . Y(n,m,f) )$.
\end{proof}
In the sequel, we sometimes identify a function $Y$ as above with its set of zeros.
The functional $\nu$ in the following theorem is called a \emph{selector}, for obvious reasons.
\begin{theorem}\label{thm6}
Let $\zeta_{\p_0}$ be a weak partial ${\rm \NCC}$-realiser. Then there is a partial functional $\nu$ taking subsets $X$ of $2^\N$ as arguments and with values in $2^\N$ such that if $X$ is closed and nonempty, then $\nu(X) \in X$.
\end{theorem}
\begin{proof}
By recursion on $n$, we use $\zeta_{\p_0}$ and primitive recursion to find (compute) a binary function $f$ such that $X \cap [\overline{f}n] \neq \emptyset$ for each $n$. Now note that $f \in X$.
\end{proof}
While seemingly basic, selectors are hard to compute as follows. 
\begin{lemma}\label{lem7} 
There is no selector  $\nu$ computable in any functional of type 2.
\end{lemma}
\begin{proof}
Assume that the selector $\nu$ is computable in $F$ and let $f,g \in 2^\N$ be distinct and  not computable in $F$. Let $X_f = \{f\}$ and $X_g = \{g\}$. 
When we compute $\nu(X_f)$ and $\nu(X_g)$ using the algorithm for $\nu$ from $F$, we will only use oracle calls for $h \in X$ for $h$ computable in $F$, and will get the same negative answer for both inputs. Thus $\nu(X_f) = \nu(X_g)$, contradicting what $\nu$ should do. 
\end{proof}
The background for this argument is treated in the proof of \cite[Lem.\ 2.14]{dagsamX}.

\begin{corollary}\label{cor8}
There is no partial ${\rm \NCC}$-realiser computable in any type 2 functional.
\end{corollary}
\begin{proof}
Follows directly from Lemma \ref{lem5}, Theorem \ref{thm6} and Lemma \ref{lem7}.
\end{proof}
Corollary \ref{cor8} can also be seen  a consequence of the fact that partial $\NCC$-realisers can deal with the computational problem (B) from the introduction.
\begin{theorem}\label{wog}
Any weak partial ${\rm \NCC}$-realiser $\zeta_{\p_0}$ can perform the following task: for $G:2^{\N}\di \N$, compute $k\in \N$ such that \textbf{there exists} a finite sub-covering of size  $k$ 
of the covering $\cup_{f\in 2^{\N}}\big[\overline{f}G(f)\big]$.
\end{theorem}
\begin{proof}
Given $G$, let $Y(k,f) = 0$ if $f = \langle f_1 , \ldots , f_k\rangle$ and the set of neighbourhoods $[\overline{f_i}G(f_i)]$ for $i = 1 , \ldots , k$ form a sub-covering of $\cup_{f\in 2^{\N}}\big[\overline{f}G(f)\big]$. Clearly, $Y$ is uniformly computable in $G$, only requiring explicit elementary constructions.  Then $\zeta_{\p_0}(Y)$ answers the computational task.
\end{proof}
On a related note, consider the following computational task (C), intermediate between (A) and (B) from Section \ref{pisec}.
A \emph{Lebesgue number} for $\cup_{f\in 2^{\N}}[\overline{f}G(f)]$ is $k\in \N$ such that $(\forall f\in 2^{\N})(\exists g\in 2^{\N})(G(g)\leq 2^{k} \wedge f\in [\overline{g}G(g)] )$.
This notion has been studied in RM in e.g.\ \cite{dagsamIII, moregusto}.
\begin{itemize}
\item[(C)] For any $G:2^{\N}\di \N$, compute a Lebesgue number for $\cup_{f\in 2^{\N}}[\overline{f}G(f)]$.
\end{itemize}
Clearly, the proof of Theorem \ref{wog} yields that partial $\NCC$-realisers can perform the task (C).  
It can be shown that (B) and (C) are equivalent, but we do not have a proof of this equivalence for $2^{\N}$ replaced by $[0,1]$.

\smallskip

Finally, we conjecture that partial $\NCC$-realisers cannot perform the computation task (A) from the introduction, i.e.\ compute the sub-covering itself, rather than just a bound on its size. 
We do not know how to establish this conjecture at the moment, and we therefore consider an `easier' problem: to show the existence of 
partial $\NCC$-realisers that do not compute $\exists^{3}$.  As discussed below Definition \ref{frag}, this easier problem can be solved by exhibiting a \emph{countably based} partial $\NCC$-realiser.  
This is the topic of Section \ref{amehoela}, where we encounter $\CH$.

\subsubsection{Partial realisers and the Continuum Hypothesis}\label{amehoela}
In this section, we show that the existence of a \emph{countably based} partial $\NCC$ realiser is equivalent to $\CH$. 

\smallskip

First, let us observe that the computational power of partial $\NCC$-realisers depends on a symbiosis with discontinuity in the form of $\exists^2$.
\begin{lemma} Assuming $\neg (\exists^2)$ there is a computable partial $\NCC$-realiser $\zeta_{\p}$\end{lemma}
\begin{proof}
Given $Y(n,m,f)$ and $n$, we search for a pair $(m,s)$ where $s$ is a binary sequence, and where $Y(n,m,s*00\dots ) = 0$. If there is an $m$ and an $f$ such that $Y(n,m,f) = 0$, the continuity of $Y$ will ensure that we find $(m,s)$ as above. We then let $\zeta_{\p}(Y)(n) = m$.
\end{proof}
There is noting dramatic about the previous lemma: the class of realisers for $\HBU$ has the same property. 
Hence, if we are interested in the relative computational power of partial $\NCC$-realisers, we may as well assume that $\exists^2$ is given.

\smallskip

Our next result is not within the scope of usual RM, but we include it in order to illustrate the special character of partial $\NCC$-realisers.
\begin{thm}\label{thm.CH} Assuming $\ZFC$, the following are equivalent:
\begin{enumerate}
\item There is a countably based partial $\NCC$-realiser $\zeta_{\p}$.
\item The continuum hypothesis $\CH$.
\end{enumerate}
\end{thm}
\begin{proof} First assume $\CH$. Define the set $\{(m_{\alpha,n},f_{\alpha,n}) : n \in \N \wedge \alpha < \aleph_1\}$ where $m_{n,\alpha} \in \N$ and $f_{n,\alpha} \in 2{^\N}$, and such that whenever $\{(m_n,f_n)\}_{n \in \N}$ is a sequence from $\N \times 2^{\N}$ there is an $\alpha < \aleph_1$ such that $m_n = m_{n,\alpha}$ and $f_n = f_{n,\alpha}$ for all $n$. We can then define $\zeta_{\p}$ by
$ \zeta_{\p}(Y)(n) = m_{n,\alpha}$ for the least $\alpha$ such that $(\forall n^{0}) (Y(n,m_{n,\alpha},f_{n,\alpha}) = 0)$.  
This $\zeta_{\p}$ will be countably based, since when terminating we only have to evaluate $Y(n,m_{n,\beta},f_{n,\beta})$ for countably many $\beta$ in order to find a suitable $\alpha$. 

\smallskip

Now assume that $\zeta_{\p}$ is a countably based partial $\NCC$-realiser. For each $f \in 2^{\N}$ let $Y_f \leq 1$ be defined by $Y_f(n,m,g) = 0$ if and only if $f =_{1} g$ and $m = f(n)$. Then $\zeta_{\p}(Y_f) = f$. Let $Z_f \subseteq Y_f$ be a countable basis for $\zeta_{\p}(Y_f)$, i.e. for all $Y$ such that $Z_f \subseteq Y$ we have that $\zeta_{\p}(Y) = f$. Let $A_f$ be the set of $g \in 2^\N$ such that $Z_f(n,m,g)$ is defined for some $n$ and $m$. 
Then $A_f$ is countable and satisfies $f \in A_f$. Indeed, otherwise $Z_f$ is a sub-function of the constant 1, and actually a sub-function of all but countably many $Y_g$. This is impossible and $\CH$ follows from:

\smallskip

\noindent {\em Claim} Let $X \subseteq 2^{\N}$ have cardinality $\aleph_1$. Then $2^{\N} = \bigcup_{f \in X}A_f.$

\smallskip

\noindent{\em Proof of Claim} Assume not, and let $g \not \in \bigcup_{f \in X}A_f$. Let $f \in X$. Since $\zeta_{\p}(Y_f) \neq \zeta_{\p}(Y_g)$ we must have that $Z_{f}$ and $Z_{g}$ are incompatible, which again means that there is a triple  $(n,m,h)$ such that both $Z_f(n,m,h)$ and $Z_g(n,m,h)$ are defined, but different. Since $h \in A_f$ and $g \not \in A_f$ by the choice of $g$ we must have that $h \neq g$, so $Y_g(n,m,h) = 1$, and consequently $Y_f(n,m,h) = 0$ (since the values differ), with the further consequence that $h = f$. Since $f \in X$ was arbitrary, this shows that $X \subseteq  A_g$, which  is impossible since $X$ is uncountable, while $A_g$ is countable. So, the assumption leads to a contradiction, and our claim follows.
\end{proof}
There are two observations to be made from this theorem. One is that in the case of $\CH$, $\exists^3$ cannot be computable in all partial $\NCC$-realisers, since $\exists^3$ is not countably based. The argument readily generalises to the case when the cardinality of the continuum is a successor cardinal, but we have no fully general proof. The argument in  case of successor cardinal is outside the scope of this paper. 

\smallskip

We conjecture that it is provable in $\ZFC$ that there is a partial $\NCC$-realiser that does not compute $\exists^3$ relative to any functional of type 2. On the other hand, if $\CH$ fails, there is no partial $\NCC$-realiser that is computable in any of the countably based functionals we have considered, like $\Theta$-functionals from Section~\ref{HCT} and the functional for non-monotone inductive definitions from \cite{dagcie18}, studied in more detail in \cite{dagnonmon}. 
We again conjecture that $\CH$ is not needed, i.e.\ the existence of such a realiser is provable in $\ZFC$.  

\smallskip

Finally, from the point of view of higher-order computability, partial $\NCC$-realisers are of interest because they are natural enough and represent a hitherto unobserved level of complexity in light of Theorem \ref{thm.CH}.  
This is clearly related to the fact that they are \emph{partial}, and in the next section we discuss
the general problem of how concepts of higher order computability extends to cases like this.

\subsubsection{Total versus partial functionals}\label{labbe}
We discuss the foundational role of partial versus total functionals via some interesting examples based on Pincherle's theorem (Example \ref{pinexa}), transfinite recursion (Example \ref{TER}), and representations of open sets (Example \ref{DEKTA}).

\smallskip

Most abstractly, given a statement of the form $(\forall x)( \exists y)\big(\Phi(x) \rightarrow \Psi(x,y)\big)$, say provable in $\ZFC$, there are two main questions of interest in computability theory.
\begin{enumerate}
\item How hard is it to compute a realiser $\zeta$ such that $(\forall x)(\Phi(x) \rightarrow \Psi(x,\zeta(x)))$?
\item What can we compute from such a realiser $\zeta$?
\end{enumerate}
\quad For item (1), the existence of a computable realiser implies that the theorem is constructively true (for some notion of `constructive').  In the non-computable case, the complexity of a realiser indicates to what extent contra-positive arguments or $\AC$ are needed. Of course, we obtain more information from a total realiser than just a partial one. 
For item (2), we get more information about an implication $A \rightarrow B$ if we can compute realisers for $B$ just from partial realisers for $A$. 

\smallskip

As to naming, we have taken the liberty to talk about `realisers', without introducing a specific realisability semantics or a precise definition of what we mean by a realiser. This is deliberate, as we want to use the expression in any situation where we have some functional that transforms information about an assumption to information about a conclusion. The main point of this section is now that: 
\begin{center}
\emph{it generally makes a huge difference whether we require our realisers to be total objects or not.} 
\end{center}
We will consider a couple of examples backing the above claim, but let us first make one point clear: combining Kleene's S1-S9 and partial functionals the way we do is not problematic or strange in the least. Indeed, it is part of the nature of computability theory that one computes partial objects, directly or relative to other objects. In our context, when we discuss computability relative to a partial object, this partial object will only take total objects as arguments, so the scheme S8 of functional composition needs no adjustment.  

\smallskip
 
As a preliminary example, in the case of total $\NCC$-realisers, it is clear from the proofs in Section \ref{totsect} that it does not matter \emph{what} the output is when the input does not satisfy the assumption of $(\forall n^{0})( \exists m^{0})( \exists f \in 2^{\N}) (Y(n,m,f) =0 )$, the computational strength, namely $\exists^{3}$, stems from the assumption \emph{that} there will always be a value.  
The following three examples are more conceptual in nature.
\begin{exa}[Pincherle's theorem]\label{pinexa}\rm
We discuss how the realisers for the original and uniform versions of Pincherle's theorem are related. 
The original version $\PIT_{o}$ was introduced in Section \ref{hijo}, while the `uniform' version $\PIT_{\u}$ is as follows:
\[
(\forall G:C\di \N)(\exists N\in \N)(\forall F:C\di \N)\big[  \LOC(F, G)\di (\forall g \in C)(F(g)\leq N)\big], 
\]
where $\LOC(F, G)$ from Section \ref{hijo} expresses that $F$ is locally bounded with $G$ realising this fact. 
These theorems were studied in detail in \cite{dagsamV}, including a reasonable definition of realiser, inspired by the work of Pincherle (\cite{tepelpinch}), as follows.

\smallskip

For the uniform version $\PIT_{\u}$, we considered \emph{Pincherle realisers} $M_{\u}^{3}$ in \cite{dagsamV} 
such that whenever $G:2^\N \rightarrow \N$ then $M_{\u}(G)$ is an upper bound for all functions $F$ satisfying $\LOC(F, G)$ from Section~\ref{hijo}.
It is shown in \cite{dagsamV} that computing an upper bound in this way amounts to the task (B) from Section \ref{pisec}.  
What is interesting is that Pincherle realisers are naturally total: $M_{\u}(F)$ must be defined for all $F$. 
By contrast, for the original version $\PIT_{o}$, a \emph{weak} Pincherle realiser $M_{o}$ has two variables: the number $M_o(F,G) = m$ is such that if $\LOC(F, G)$
then $F$ is bounded by $m$ on $2^\N$. 
Even though we considered total functionals $M_{o}$ in \cite{dagsamV}, we do not need $M_o(F,G)$ to be defined unless $\LOC(F, G)$ is satisfied, so here it is equally natural to consider a \emph{partial} realiser. 
In fact, the proof of \cite[Cor.~3.8]{dagsamV} mentioned in Example~\ref{TER} can be adjusted to show that any \emph{partial weak} Pincherle realiser $M_{o}$, together with $\exists^2$, computes a total realiser for transfinite recursion.
This proof is however beyond the scope of this paper. 

\smallskip

Moreover, computing a Pincherle realiser from a \emph{partial} $\NCC$-realiser via Theorem~\ref{wog} demonstrates how uniform Pincherle's theorem $\PIT_{\u}$ can be proved from $\NCC$ (assuming $\HBU$), while computing it from a total $\NCC$-realiser just witnesses that the theorem is provable in $\Z_2^\Omega$.  A similar observation holds for $\PIT_{o}$ and $\WKL$.
\end{exa}
The next example deals with transfinite recursion and Pincherle realisers. 
\begin{exa}[Transfinite recursion]\label{TER}\rm
We assume there is a partial functional $\Gamma$ such that if $(X,\prec)$ is a well-ordering of a subset of $\N$ and $F:\N^\N \rightarrow \N^\N$, then $\Gamma(X,\prec ,F)$ is a sequence of functions $f_x$ such that if $x \in X$ then $f_x = F(f_{\prec x})$, where $f_{\prec x}(\langle y , z\rangle) = f_y(z)$ if $y \prec x$ and 0 otherwise. 
This gives rise to three different functionals of increasing power, as follows.
\begin{itemize}
\item If we are satisfied with $\Gamma$ being partial, it is outright computable by using the recursion theorem for S1-S9. 
\item If we want a total extension of $\Gamma$, but it does not matter what the value is when $(X,\prec)$ is \emph{not} a well-ordering, we can use $\exists^2$ and a Pincherle realiser $M_{\u}$ to compute such a total extension by \cite[Cor.\ 3.8]{dagsamV}. 
\item Given a $\Theta$-functional and $\exists^2$, we may expand $\Gamma$ so that it extracts an infinite descending sequence in $(X, \prec)$ when $\Gamma$ does not provide a fixed point to the recursion equation for iterating $F$ along $(X,\prec)$ (see \cite[Cor.\ 3.16]{dagsamII}). 
\end{itemize}
An important aspect of these three results is the difference in what we mean by `computability'. In the first case, we use the full power of Kleene-computability, and the full use of S1-S9 will only make sense assuming principles of transfinite recursion anyhow, so there is not much insight to be gained from this. For the other two cases, we only use a fragment of G\"odel's $ T$, and thereby illustrate the computational power of compactness in various guises.
\end{exa}
Another example is provided by the $\Delta$-functional introduced in \cite{dagsamVII}*{\S7}. 
Note that modulo $\exists^{2}$, (R.3) below is exactly the usual `countable union of open balls' representation of open sets from RM, called (R.4) in \cite{dagsamVII} and introduced in \cite{simpson2}*{II}.
\begin{exa}[Representations of open sets]\label{DEKTA}\rm
The $\Delta$-functional outputs a `high-level' representation (R.3) of an open set $O$ from a `low-level' representation (R.2) of $O$, as defined below the following two clauses.
\begin{itemize}
\item[(R.2)] If $O \subseteq [0,1]$ is open, an (R.2)-representation is a function $Y:[0,1] \rightarrow \R$ such that $x \in O \leftrightarrow Y(x) > 0$ and moreover such that if $Y(x) > 0$ then $(x - Y(x) , x + Y(x)) \cap [0,1] \subseteq O$.
\item[(R.3)] If $O \subseteq [0,1]$ is open, the (R.3)-representation is the continuous function $Y'$ where $Y'(x)$ is the distance from $x$ to $[0,1] \setminus O$, where the distance to the empty set is defied as 1.
\end{itemize}
With $Y$ and $ Y'$ as in (R.2) and (R.3), we have that $\Delta(Y) = Y'$.  The functional $\Delta$ is of low complexity among the genuine type 3 functionals; it is however unknown what happens with the complexity if we extend $\Delta$ to a total object.  Indeed, the point is that if $\Delta$ can be partial, we never (have to) specify what to do if the input does not represent an open set at all.  Hence, when we say that $\Delta$ is computable from a Pincherle realiser $M_{\u}$ (see \cite{dagsamVI}*{Theorem 7.5}), the algorithm works under the assumption that the input is an (R.2)-representation of an open set. In this case, $\Delta$ is also computable from a partial $\NCC$-realiser $\zeta_{\p}$ and $\exists^2$ as well.
\end{exa}

\subsubsection{Alternatives to Kleene computability}\label{optional}
We briefly discuss the possibility of using computational frameworks other than Kleene's S1-S9.  

\smallskip

On one hand, we have seen that the partial functional for transfinite recursion is outright S1-S9 computable. On the other hand, the step from $\ACA_0$ to $\ATR_0$ is a significant step in logical strength. 
The explanation is of course that the assumption 
\begin{center}
\emph{the definition of computability via S1-S9 is sound }
\end{center}
is itself quite strong.  In fact, this soundness goes beyond the strength of transfinite recursion, as it involves the termination of monotone inductions (see \cite{dagcie18}). 
We will not pursue this discussion here, or make any precise mathematical claims related to it, but let us emphasise the following observation.  

\smallskip

On one hand, positive computability results are more interesting when the concept of higher-order computability at hand is (far) simpler than full S1-S9.  A natural such simple framework is finite type theory with constants for the arithmetical operations and the partial $\mu$-operator. Note that in the proof of Theorem~\ref{thm6}, we go slightly beyond this, but generally our positive results are witnessed by terms in G\"odel's $T$ of low complexity.

\smallskip

On the other hand, \emph{non-computability} results are better the stronger the concept of relative computability involved is.  In this case, S1-S9 is of great interest. 
In fact, our non-computability results generally make use of S1-S9, while \emph{infinite time Turing machines} would be too strong (see \cite{dagnonmon}).  
Finally, as explored systematically in \cite{dagsamX}, it should be noted that computability theory based on S1-S9 is a crucial tool in constructing models for fragments of $\Z_2^\Omega$, 
as in e.g. \cites{dagsam, dagsamII, dagsamIII,dagsamV, dagsamVI, dagsamVII, dagsamX}.

\subsection{Turing machines and higher types}
We finish this paper with a section on accommodating higher types in Turing's framework.

\smallskip

Now, Turing's famous `machine' framework (\cite{tur37}) introduces an intuitively convincing concept of `computing with real numbers'.  
Certain higher type objects, like continuous functions on $\R$, can be represented as real numbers, but this `coding' is not without its problems (see \cite{dagsamVI, samrep})
By contrast, Kleene's S1-S9 has the advantage of providing a notion of `computing with objects of finite type', at the cost of the simplicity of Turing's framework, like e.g.\ the lack of a counterpart of Kleene's $T$-predicate or the axiomatic encoding of the recursion theorem in S9.    

\smallskip

It is then a natural question whether we can discuss certain higher-order results in terms of Turing computability.  An example from \cite{samph}*{\S3.2.1} is as follows: 
let `$\leq_{T}$' be the usual Turing reducibility relation and let $J(Y)$ be the set $\{n\in \N: (\exists f\in \N^{\N})(Y(f, n)=0)\}$, i.e.\ the set $X$ claimed to exist by $\BOOT$.  Now, $\BOOT$ follows from the monotone convergence theorem \emph{for nets} indexed by Baire space in $[0,1]$ by \cite{samph}*{Theorem 3.7}.  This implication yields the following:
\[\textstyle
\textup{for any $Y^{2}$, there is a net $x_{d}:D\di [0,1] $ such that $x=\lim_{d}x_{d}$ implies $ J(Y)\leq_{T}x$.}
\]
Note that the net $x_{d}:D\di [0,1]$ can be defined in terms of $Y^{2}$ via a term of G\"odel's $T$.  A similar result for the Baire category theorem can be found in \cite{dagsamVI}*{\S6.2.2}.

\smallskip

We now discuss a similar result based on $[\WKL+\NCC]\di \PIT_{o}$ as in Theorem~\ref{hijot}.  The computational properties of $\WKL$ in Turing's framework are (very) well-studied, 
and the aforementioned implication suggests the possibility of studying Pincherle's theorem in the same way, namely as follows.  

\smallskip

Given $Y$ as in $\NCC$, define $C(Y)$ as the function $g$ therein, i.e.\ $m=C(Y)(n)$ yields $(\exists f\in 2^{\N})(Y(f, n, m)=0)$.  
Then clearly we have $C(Y)\leq_{T} J(Y)$, where we assume the two number variables are coded into one.  Now consider the contraposition of Pincherle's theorem (without realisers): 
\begin{center}
if a functional $F$ is unbounded on $2^{\N}$, there is 
a point $x_{0}\in 2^{\N}$ such that $F$ is unbounded on all its neighbourhoods.    
\end{center}
Similar to the above, $\exists^{3}$ can (S1-S9) compute $x_{0}$ in terms of $F$, but no type two functional can.  
However, we can state the following:
\begin{center}
if a functional $F$ is unbounded on $2^{\N}$, there is 
a point $x_{0}\leq_{T} (C(F_{0}))'$ in $2^{\N}$ such that $F$ is unbounded on all its neighbourhoods.    
\end{center}
Note that $(C(F_{0}))'$ is the Turing jump of $C(F_{0})$, which is well-defined.   The exact definition of $F_{0}$ is of course based on the formula in square brackets in \eqref{worng}, with obvious/minimal coding.
Clearly, we could apply $\QFAC^{0,1}$ to `$F$ is unbounded on $2^{\N}$' and the jump of the resulting sequence would also yield a point like $x_{0}$.  
However, the point thus obtained is not Turing computable from e.g.\ oracles provided by $J$.  

\smallskip

Finally, the same can be established for the contraposition of $\HBC$ \emph{mutatis mutandis} and many similar theorems about open sets as in Definition \ref{openset}.

\begin{ack}\rm
We thank Anil Nerode for his helpful suggestions.
Sam Sanders' research was supported by the \emph{Deutsche Forschungsgemeinschaft} via the DFG grant SA3418/1-1. 
\end{ack}

\bye